\documentclass[11pt,letter]{article}

\everymath{\color{MidnightBlack}}

\usepackage{titling}
\thanksmarkseries{arabic}
\newcommand*\samethanks[1][\value{footnote}]{\footnotemark[#1]}

\RequirePackage[T1]{fontenc}
\RequirePackage[utf8]{inputenc}
\usepackage[greek,russian,english]{babel}
\usepackage{alphabeta}

\usepackage[pdftex,
backref=true,
plainpages = false,
pdfpagelabels,
hyperfootnotes=true,
pdfpagemode=FullScreen,
bookmarks=true,
bookmarksopen = true,
bookmarksnumbered = true,
breaklinks = true,
hyperfigures,
pagebackref,
urlcolor = magenta,
urlcolor = MidnightBlack,
anchorcolor = green,
hyperindex = true,
colorlinks = true,
linkcolor = black!30!blue,
citecolor = black!30!green
]{hyperref}

\usepackage[svgnames]{xcolor}
\usepackage{bm,fullpage,aliascnt,amsthm,amsmath,amsfonts,rotate,datetime,amsmath,ifthen,eurosym,wrapfig,cite,url,subcaption,cite,amsfonts,amssymb,ifthen,color,wrapfig,rotate,lmodern,aliascnt,datetime,graphicx,algorithmic,algorithm,enumerate,enumitem,todonotes,fancybox,cleveref,bm,subcaption,tabularx,colortbl,xspace,graphicx,algorithmic,algorithm}
\usepackage{stmaryrd,lmodern}

\newcommand{\greenfbox}[1]{\green{\fbox{\black{#1}}}}
\newcommand{\tw}{{\sf tw}}
\newcommand{\sbn}{{\sf sbn}}
\newcommand{\ltp}{{\sf ltp}}

\newcommand{\hh}{\end{document}}

\newcommand{\remove}[1]{}

\newcommand{\NP}{{\sf NP}\xspace}

\newcommand{\obs}{{\sf Obs}}
\newcommand{\exc}{{\sf exc}}

\RequirePackage{stmaryrd}
\usepackage{textcomp}
\DeclareUnicodeCharacter{2286}{\subseteq}
\DeclareUnicodeCharacter{2192}{\ifmmode\to\else\textrightarrow\fi}
\DeclareUnicodeCharacter{2203}{\ensuremath\exists}
\DeclareUnicodeCharacter{183}{\cdot}
\DeclareUnicodeCharacter{2200}{\forall}
\DeclareUnicodeCharacter{2264}{\leq}
\DeclareUnicodeCharacter{2265}{\geq}
\DeclareUnicodeCharacter{8614}{\mathbin{\to }}
\DeclareUnicodeCharacter{8656}{\Leftarrow}
\DeclareUnicodeCharacter{8657}{\Uparrow}
\DeclareUnicodeCharacter{8658}{\Rightarrow}
\DeclareUnicodeCharacter{8659}{\Downarrow}
\DeclareUnicodeCharacter{8669}{\rightsquigarrow}
\newcommand{\eqdef}{\stackrel{{\scriptsize\rm def}}{=}}
\DeclareUnicodeCharacter{8797}{\eqdef}
\DeclareUnicodeCharacter{8870}{\vdash}
\DeclareUnicodeCharacter{8873}{\Vdash}
\DeclareUnicodeCharacter{22A7}{\models}
\DeclareUnicodeCharacter{9121}{\lceil}
\DeclareUnicodeCharacter{9123}{\lfloor}
\DeclareUnicodeCharacter{9124}{\rceil}
\DeclareUnicodeCharacter{2208}{\in}
\DeclareUnicodeCharacter{9126}{\rfloor}
\DeclareUnicodeCharacter{9655}{\triangleright}
\DeclareUnicodeCharacter{9665}{\triangleleft}
\DeclareUnicodeCharacter{9671}{\diamond}
\DeclareUnicodeCharacter{9675}{\circ}
\DeclareUnicodeCharacter{10178}{\bot}
\DeclareUnicodeCharacter{10214}{} 
\DeclareUnicodeCharacter{10215}{} 
\DeclareUnicodeCharacter{10229}{\longleftarrow}
\DeclareUnicodeCharacter{10230}{\longrightarrow}
\DeclareUnicodeCharacter{10231}{\longleftrightarrow}
\DeclareUnicodeCharacter{10232}{\Longleftarrow}
\DeclareUnicodeCharacter{10233}{\Longrightarrow}
\DeclareUnicodeCharacter{10234}{\Longleftrightarrow}
\DeclareUnicodeCharacter{10236}{\longmapsto}
\DeclareUnicodeCharacter{10238}{\Longmapsto} 
\DeclareUnicodeCharacter{10503}{\to }    
\DeclareUnicodeCharacter{10971}{\mathrel{\not\hspace{-0.2em}\cap}}
\DeclareUnicodeCharacter{65294}{\ldotp}
\DeclareUnicodeCharacter{65372}{\mid}

\usepackage{color}
\definecolor{MidnightBlack}{rgb}{0.1,0.1,.34}
\definecolor{MidnightBlue}{rgb}{0.1,0.1,0.43}
\definecolor{Black}{rgb}{0,0, 0}
\definecolor{Blue}{rgb}{0, 0 ,1}
\definecolor{Red}{rgb}{1, 0 ,0}
\definecolor{White}{rgb}{1, 1, 1}
\definecolor{grey}{rgb}{.6, .6, .6}
\definecolor{Mygreen}{rgb}{.0, .7, .0}
\definecolor{Yellow}{rgb}{.55,.55,0}
\definecolor{Mustard}{rgb}{1.0, 0.86, 0.35}
\definecolor{applegreen}{rgb}{0.55, 0.71, 0.0}
\definecolor{darkturquoise}{rgb}{0.0, 0.81, 0.82}
\definecolor{celestialblue}{rgb}{0.29, 0.59, 0.82}
\definecolor{green_yellow}{rgb}{0.68, 1.0, 0.18}
\definecolor{crimsonglory}{rgb}{0.75, 0.0, 0.2}
\definecolor{darkmagenta}{rgb}{0.30, 0.0, 0.30}
\definecolor{internationalorange}{rgb}{1.0, 0.31, 0.0}
\definecolor{darkorange}{rgb}{1.0, 0.55, 0.0}
\definecolor{ao}{rgb}{0.0, 0.5, 0.0}
\definecolor{awesome}{rgb}{1.0, 0.13, 0.32}
\definecolor{darkcyan}{rgb}{0.0, 0.50, 0.50}
\definecolor{violet}{rgb}{0.93, 0.51, 0.93}
\definecolor{brown}{rgb}{0.65, 0.16, 0.16}
\definecolor{orange}{rgb}{1.0, 0.65, 0.0}

\newcommand{\black}[1]{{\color{Black}#1}}

\newcommand{\blue}[1]{{\color{Blue}#1}}
\newcommand{\red}[1]{{\color{Red}#1}}
\newcommand{\green}[1]{{\color{Mygreen}#1}}

\newcommand{\violet}[1]{{\color{violet}#1}}

\usepackage{tikz}
\usetikzlibrary{ipe}
\usetikzlibrary{patterns}
\usetikzlibrary{calc,3d}
\usetikzlibrary{intersections,decorations.pathmorphing,shapes,decorations.pathreplacing,fit}
\usetikzlibrary{shapes.geometric,hobby,calc}
\usetikzlibrary{decorations}
\usetikzlibrary{decorations.pathmorphing}
\usetikzlibrary{decorations.text}
\usetikzlibrary{shapes.misc}
\usetikzlibrary{decorations,shapes,snakes}

\newcounter{func}

\newcommand{\funref}[1]{\hyperref[#1]{f_{\ref*{#1}}}} 
\newcounter{con}

\newcommand{\conref}[1]{\hyperref[#1]{c_{\ref*{#1}}}} 

\newcommand{\mynewtheorem}[2]{
	\newaliascnt{#1}{dummy}
	\newtheorem{#1}[#1]{#2}
	\aliascntresetthe{#1}
}

\theoremstyle{plain}
\mynewtheorem{theorem}{Theorem}
\mynewtheorem{@menospreciando}{Proposition}
\mynewtheorem{corollary}{Corollary}
\mynewtheorem{lemma}{Lemma}
\theoremstyle{definition}
\mynewtheorem{definition}{Definition}
\theoremstyle{remark}
\mynewtheorem{sublemma}{Sublemma}
\mynewtheorem{claim}{Claim}
\mynewtheorem{remark}{Remark}
\mynewtheorem{fact}{Fact}
\mynewtheorem{conjecture}{Conjecture}
\mynewtheorem{question}{Question}
\mynewtheorem{observation}{Observation}
\mynewtheorem{problem}{Problem}
\mynewtheorem{note}{Note}

\addto\extrasenglish{

}

\makeatletter
\newtheoremstyle{caja1}
  {\topsep}
  {\topsep}
  {\itshape}
  {}
  {}
  {}
  {.5em}
  {\blue{\fbox{\black{\thmname{#1}~\thmnumber{#2}\@ifempty{#3}{.}{}\thmnote{ (#3).}}}}}
\makeatother
\theoremstyle{caja1}

\makeatletter
\newtheoremstyle{caja2}
  {\topsep}
  {\topsep}
  {\itshape}
  {}
  {}
  {}
  {.5em}
  {\green{\fbox{\black{\thmname{#1}~\thmnumber{#2}\@ifempty{#3}{.}{}\thmnote{ (#3).}}}}}
\makeatother
\theoremstyle{caja2}

\makeatletter
\newtheoremstyle{caja3}
  {\topsep}
  {\topsep}
  {\itshape}
  {}
  {}
  {}
  {.5em}
  {\red{\fbox{\black{\thmname{#1}~\thmnumber{#2}\@ifempty{#3}{.}{}\thmnote{ (#3).}}}}}
\makeatother
\theoremstyle{caja3}

\newtheorem{innercustomgeneric}{\customgenericname}
\providecommand{\customgenericname}{}

\tikzstyle{every node}=[circle, draw, fill=black,
                        inner sep=0pt, minimum width=4pt]
                        
\usepackage[font=small,labelfont=bf]{caption}

\begin{document}

\title{On Strict Brambles\thanks{The second and the third author were supported  by   the ANR projects DEMOGRAPH (ANR-16-CE40-0028), ESIGMA (ANR-17-CE23-0010), and the French-German Collaboration ANR/DFG Project UTMA (ANR-20-CE92-0027).}\medskip\medskip\medskip}

\author{ 
Emmanouil Lardas%
\thanks{Department of Mathematics, National and Kapodistrian University of Athens, Athens, Greece.}
\and 
Evangelos Protopapas%
\thanks{LIRMM, Univ Montpellier, CNRS, Montpellier, France.} 
\and
Dimitrios  M. Thilikos\samethanks[3]
\and\\
Dimitris Zoros\samethanks[2]
}
\date{}

\maketitle

\begin{abstract}

\noindent A {\em strict bramble} of a graph $G$ is a collection of pairwise-intersecting connected subgraphs of $G.$ The {\em order} of a strict bramble ${\cal B}$ is the minimum size of a set of vertices intersecting all sets of ${\cal B}.$  The {\em strict bramble number} of $G,$ denoted by ${\sf sbn}(G),$ is the maximum order of a strict bramble in $G.$ The strict bramble number of $G$ can be seen as a way to extend the notion of acyclicity, departing from the fact  that  (non-empty) acyclic graphs are exactly the graphs where every strict bramble has order one.  We initiate the study of this graph parameter by providing three alternative definitions, each revealing different structural characteristics.   The first is a min-max theorem asserting that ${\sf sbn}(G)$ is equal to the minimum $k$ for which $G$ is a minor of the lexicographic product of a tree and a clique on $k$ vertices (also known as the {\em lexicographic tree product number}). The second characterization is in terms of a new variant of a tree decomposition called {\em lenient tree decomposition}. We prove that ${\sf sbn}(G)$ is equal to the minimum $k$ for which there exists a lenient tree decomposition of $G$ of width at most $k.$ The third characterization is in terms of extremal graphs. For this, we define, for each $k,$ the concept of a {\em $k$-domino-tree} and we prove that every edge\mbox{-}maximal graph of strict bramble number at most $k$ is a $k$-domino-tree. We also identify three graphs that constitute the minor-obstruction set of the class of graphs with strict bramble number at most two. We complete our results by proving that, given some $G$ and $k,$  deciding whether ${\sf sbn}(G) \leq k$ is an {\sf NP}-complete problem. \end{abstract}
\medskip\medskip\medskip

\newpage
\section{Introduction}

A well-known definition of acyclicity is the following: a non-empty graph $G$ is acyclic if for every collection of pairwise intersecting subtrees of $G$ there is some vertex appearing in every subtree. In this paper we deal with a natural parametric extension of acyclicity, that is, the minimum $k$ such that for every collection of pairwise intersecting subtrees of $G$ there is a set of $k$ vertices intersecting all of them. To our knowledge, this graph parameter\footnote{We use the term {\em graph parameter} for every function mapping  graphs to non-negative integers.} appeared for the first time by Kozawa, Otachi and Yamazaki in~\cite{KozawaOY14lower} with the name {\sl PI number} (where {PI} stands for ‘‘Pairwise Intersecting’’) and was used in order to derive lower bounds for the treewidth of several classes of product graphs (for the definition of treewidth, see~\autoref{@imprecaciones}). The same parameter was recently introduced by Aidun, Dean, Morrison, Yu, and  Yuan in~\cite{AidunDMYY19treew} with the name {\sl strict bramble number} and is the term that we adopt in this paper. The strict bramble number was used in~\cite{AidunDMYY19treew} in order to study the relation of treewidth and the gonality on particular classes of graphs.
 
 \paragraph{Strict brambles.}
We proceed with the formal definition of the {\em strict bramble number}. Given a collection ${\cal B}$ of vertex sets of $G$  and some vertex set $X,$ we say that $X$ {\em covers} ${\cal B}$ if every set in ${\cal B}$ has some vertex in common with $X.$   We say that a vertex set $S$ is {\em connected} if the subgraph of $G$ induced by $S$ is connected. 
A {\em strict bramble} of a graph $G$ is a  collection ${\cal B}$ of vertex sets of $G$  such that:
\begin{enumerate}[label=(\arabic*)]
\item every set in ${\cal B}$ is connected;
\item\label{@konstruieren} every two sets in ${\cal B}$ have some vertex in common. 
\end{enumerate}
The {\em order} of a strict bramble ${\cal B}$ of $G$ is the minimum size of a set that covers ${\cal B}$ and is denoted by ${\sf order}({\cal B}).$  The {\em strict bramble number} of $G,$ denoted by ${\sf sbn}(G),$ is the maximum order of a strict bramble of $G.$

\paragraph{Brambles.}
Given two vertex sets $S$ and $S'$ of a graph $G$ we say that $S$ and $S'$ {\em touch} in $G$ if either they have some vertex in common or there is an edge with one endpoint in $S$ and the other in $S'.$ If we relax the definition of strict bramble by substituting \ref{@konstruieren} with:
\begin{enumerate}
\item[($2'$)]\label{@discreditable} every two sets in ${\cal B}$ are touching,
\end{enumerate}
then we obtain the (classic) notion of {\em bramble} and the parameter {\em bramble number}, denoted by ${\sf bn}(G),$ introduced by Seymour and Thomas in~\cite{SeymourT93graph}\footnote{We wish to stress that in~\cite{SeymourT93graph}  the term ``{\sl screen}'' was used, instead of the term ``bramble''.}. The study of brambles attracted a lot of attention because of the main result in~\cite{SeymourT93graph}, that is a min-max theorem asserting that for every graph $G,$ the treewidth of $G$ is one less than its bramble number. As already observed in~\cite{KozawaOY14lower} (using the results of \cite{Reed97anewc}), for every graph $G,$ it holds that ${\sf sbn}(G) \leq {\sf bn}(G) \leq 2 \cdot {\sf sbn}(G)$ which, in turn, implies that:
\begin{eqnarray}
\sbn(G)-1 \leq \tw(G)  \leq  2\cdot \sbn(G)-1.\label{@marvellously}
\end{eqnarray}
{\sl Treewidth}, the min-max analogue of brambles, is one of the most important graph parameters. It was introduced by Robertson and Seymour in~\cite{RobertsonS84GMIII} (see~\cite{BerteleB72nonser,Halin76sfun} for earlier appearances). Treewidth served as a cornerstone parameter of the Graph Minors series of Robertson and Seymour and is omnipresent in a wide range of topics in combinatorics and in graph algorithms~\cite{Bodlaender98apart}.

In this paper we initiate the study of the strict bramble number, mainly motivated by the fact that, so far, no min-max analogue, parallel to treewidth, is known for this graph parameter. In this direction, we provide three alternative definitions of the strict bramble number, each revealing different characteristics of this parameter. We start with a brief introduction of these definitions.

\paragraph{Lexicographic tree product.} Let $G, H$ be a pair of graphs. The {\em lexicographic product} of $G$ and $H,$ denoted by $G \cdot H,$ is the graph whose vertex set is the Cartesian product of the vertex sets of $G$ and $H$ and where the vertex $(u, v)$ is adjacent with the vertex $(w, z)$ in $G \cdot H$ if and only if either $u$ is adjacent with $w$ in $G$ or it holds that $u = w$ and $v$ is adjacent with $z$ in $H.$ The {\em lexicographic tree product number} of $G$ is defined by Harvey and  Wood in~\cite{HarveyW17param} as:
\begin{eqnarray*}
\ltp(G)  & = & \min\{ k \in \mathbb{N} \mid \textrm{there is a tree $T$ such that $G$ is a minor of $T \cdot K_{k}$}\}.\label{@intelligiblement}
\end{eqnarray*}
(For the definition of the minor relation, see \autoref{@desasosiegos}). Our first contribution is to show that the lexicographic tree product number and the strict bramble number are the same parameter. Incidentally, replacing {\sf sbn} with {\sf ltp} in \eqref{@marvellously}, was already proved in~\cite{HarveyW17param}.

\paragraph{Lenient tree decompositions.}
Let $G$ be a graph, $T$ a tree and let $χ$ be a function mapping vertices of $T$ to vertex sets of $G.$ We say that two vertices $t,t'$ of $T$  are {\em close} in $T$ if either they
are identical or they are adjacent.
The pair $(T, χ)$ is a \textit{lenient tree decomposition} of $G$ if it satisfies the following three conditions:
\begin{enumerate}[label=(C\arabic*)]
    \item\label{@predominantly} $\bigcup_{t \in V(T)} χ(t)$ is the vertex set of $G$;
    \item\label{@calamitously} for every edge $e$ of $G,$ there are two close vertices $t,t'$ of $T$ such that, $e \subseteq χ(t) \cup χ(t^{\prime})$;
    \item\label{@unpretentious} for every vertex $x$ of $G,$ the set $\{ t \mid x \in χ(t)\}$ is connected in $T.$
\end{enumerate}

We define the {\em width} of $(T, χ),$ as the maximum size of a set $χ(t),$ for vertices $t$ of $T.$ See \autoref{lenient_dec} for an example of the above definition. Our second characterization of the strict bramble number is that, for every graph $G,$ $\sbn(G)$ is equal to the minimum width of a lenient tree decomposition of $G.$ In that way, lenient tree decompositions can serve as the  analogue of tree decompositions for the case of strict brambles. Notice that the definition of a tree decomposition, given in~\autoref{@desasosiegos}, follows from the above definition if we substitute ``close'' by ``identical''.

\begin{figure}[t]
\begin{center}
\graphicspath{{./Figures/}}
\scalebox{1.2}{\includegraphics{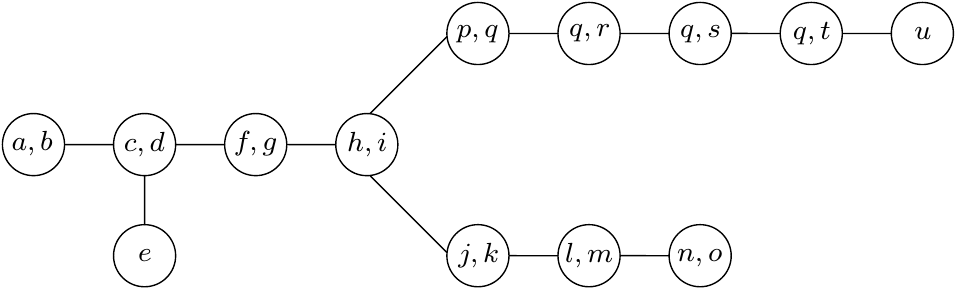}}
\end{center}
\caption{A lenient tree decomposition of the graph $G'$ of \autoref{domino_expl}.}
\label{lenient_dec}
\end{figure}

\paragraph{$k$-domino-trees.} Given a non-negative integer $k,$ a {\em $k$-tree} is recursively defined as follows: a graph $G$ is a {\em $k$-tree} if it is either isomorphic to $K_{r},$ for some $r\leq k,$ or it 
contains a vertex $v$ of degree $k$ in $G$ whose neighborhood induces a clique in $G$ and whose removal from $G$ yields a $k$-tree. It is known that among all the graphs with treewidth at most $k,$ those that are 
edge-maximal (that is, after the addition of any edge they obtain treewidth more than $k$) are precisely the $k$-trees. 
This implies that the treewidth of a graph can be defined as the minimum $k$ for which $G$ is a spanning subgraph of a $k$-tree. Is there an analogous definition for the strict bramble number? What are the edge-extremal graphs of strict bramble number at most $k$?

Our third characterization is obtained by answering the above  questions. For this, in \autoref{@monosyllabic}, we introduce the concept of a $k$-domino-tree. We prove that $\sbn(G)$ is equal to the minimum $k$ for which $G$ is a spanning subgraph of a $k$-domino-tree.

The proof of all aforementioned equivalences is given in \autoref{@comprehension}. Moreover, in~\autoref{@preoccupations}, we prove that the edge-extremal graphs of strict bramble number at most $k$ are precisely the  $k$-domino-trees. Interestingly, $k$-domino-trees enjoy a more elaborate structure than the one of $k$-trees. While all $k$-trees on $n$ vertices have the same number of edges the same does not hold for the $k$-domino-trees on $n$ vertices. As we see in~\autoref{@preoccupations} the number of edges may vary considerably.

\paragraph{Obstructions.} Given a graph class ${\cal G},$ the minor-obstruction set of ${\cal G},$ denoted by $\obs({\cal G}),$ is the set of all minor-minimal graphs that do not belong to ${\cal G}$ (for the definition of the minor relation, see \autoref{@desasosiegos}). In the case where ${\cal G}$ is closed under taking minors -- that is, minors of the graphs in ${\cal G}$ are also in ${\cal G}$ -- then $\obs({\cal G})$ offers an exact characterization of ${\cal G}$ as, for every graph $G,$ $G\in {\cal G}$ iff for every $H \in {\sf Obs}({\cal G}),$ $H$ is not a minor of $G.$ Moreover, this gives a ``finite'' characterization of ${\cal G}$ as, by the Robertson and Seymour theorem \cite{RobertsonS04GMXX}, $\obs({\cal G})$ is always a finite set. Let ${\cal G}_{k}$ be the class of graphs of strict bramble number at most $k.$ As we already mentioned, ${\cal G}_{1}$ is the class of all acyclic graphs, therefore $\obs({\cal G}_{1})=\{K_{3}\}.$  Our next result is the identification of  $\obs({\cal G}_{2}),$ that consists of the three graphs depicted in \autoref{fig_obs}.

\paragraph{{\sf NP}-completeness.} We complete our study by showing that the problem of deciding whether, given a graph $G$ and a non-negative integer $k,$ $\sbn(G) \leq k,$ is an \NP-complete problem. We do so by reducing the computation of treewidth to the computation of the strict bramble number (\autoref{@unflattering}). Notice that, membership in \NP is non-trivial for this problem. For this, our min-max equivalence result (\autoref{@comprehension}) comes in handy. The paper concludes with some open problems on the strict bramble number, presented in \autoref{@mechanisches}.

\section{Preliminaries}
\label{@desasosiegos}

\paragraph{Integers, sets, and tuples.}
We denote by $\Bbb{N}$ the set of non-negative integers. Given two integers $p$ and $q,$ the set $[p,q]$ refers to the set of every integer $r$ such that $p ≤ r ≤ q.$ For an integer $p≥ 1,$ we set $[p]=[1,p].$ 
For a set $S,$ we denote by $2^{S}$ the set of all subsets of $S$ and, given an integer $r∈[|S|],$  we denote by $\binom{S}{r}$ the set of all subsets of $S$ of size $r.$ 

\paragraph{Graphs.} All graphs in this paper are simple, i.e. they are finite and they do not have multiple edges or loops. Given a graph $G$ we denote its vertex and edge set by $V(G)$ and $E(G)$ respectively. Given graphs $H$ and $G,$ $H$ is a subgraph of $G,$ denoted as $H \subseteq G$ if, $V(H) \subseteq V(G)$ and $E(H) \subseteq E(G).$ Given an $S\subseteq V(G),$ we define the subgraph of $G$ induced by $S,$ denoted by $G[S],$ as the graph $G[S] = \big(S,{S\choose 2}\cap E(G)\big).$ Also, we define $G- S = G[V(G)\setminus S]$ and $G - u = G - \{ u \},$ for $u \in V(G).$

\paragraph{Basic definitions.}  Given a graph $G,$ we denote by ${\sf cc}(G),$ the {\em set of all connected components} of $G.$ Let $S\subseteq V(G).$ We define the {\em connectivity-degree} of $S$ as ${\sf cdeg}_{G}(S)=|{\sf cc}(G - S)|,$ i.e. the number of connected components of $G - S.$ We also define ${\sf Acc}(G,S)=\{G[V(C) \cup S]\mid C\in {\sf cc}(G - S)\}$ and we call ${\sf Acc}(G,S)$ the {\em set of all augmented connected components} of $G - S.$ Let $x \in V(G).$ We denote the {\em degree} of $x \in V(G),$ by ${\sf deg}_{G}(x).$ Also for $A,B \subseteq V(G)$ we say that $A$ and $B$ {\em intersect} if $A \cap B \neq \emptyset.$ Given a tree $T$ and some $t \in V(T)$ with ${\sf deg}_{T}(t) = 1,$ we call $t$ a {\em leaf} node of $T.$ Otherwise, we call $t$ an {\em internal} node of $T.$ We also use $K_{k}$ to denote the complete graph on $k$ vertices.

\paragraph{Paths and separators.} Let $x, y \in V(G).$ If $x = y,$ an $x\text{-}y$ path of $G$ is the graph $(\{x\}, \emptyset).$ Otherwise, an $x\text{-}y$ path $P$ of $G$ is any connected subgraph of $G,$ where ${\sf deg}_{P}(x) = 1$ and ${\sf deg}_{P}(y) = 1$ and for any other vertex $z \in V(P),$ ${\sf deg}_{P}(z) = 2.$ The {\em length} of a path $P$ is equal to $|E(P)|.$ The {\em distance} between $x$ and $y$ in $G$ is the minimum number of edges of an $x\text{-}y$ path in $G.$ Given a path $P$ and a vertex $z \in V(P),$ we call $z$ an {\em internal} vertex of $P$ if ${\sf deg}_{P}(z) = 2,$ otherwise we call $z$ a {\em terminal} vertex of $P.$ We call two paths $P$ and $Q$ {\em vertex disjoint} if $V(P) \cap V(Q) = \emptyset$ and {\em internally vertex disjoint} if $V(P) \cap V(Q)$ contains only terminal vertices. Let $X, Y \subseteq V(G).$ An $X\text{-}Y$ path $P$ of $G$ is an $x\text{-}y$ path of $G,$ where $x \in X$ and $y \in Y.$

Let $x,y \in V(G).$ A set $S \subseteq V(G)$ is an $(x,y)$-separator of $G$ if $x$ and $y$ are in different connected components of $G - S.$ $S$ is a {\em minimal} $(x,y)$-separator if none of its proper subsets is an $(x,y)$-separator. A set $S$ is a {\em minimal separator} of $G$ if it is a minimal $(x,y)$-separator for some $x, y \in V(G).$ Also let $X, Y \subseteq V(G).$ A set $S \subseteq V(G)$ is an $(X, Y)$-separator of $G,$ if it is an $(x, y)$-separator of $G,$ for every pair $x,y$ of vertices $x \in X \setminus S,$ $y \in Y \setminus S.$  A graph is $k$-connected, if it contains at least $k+1$ vertices and does not contain any $(x,y)$-separator of less than $k$ vertices.

\begin{lemma}\label{@constellations} Let $G$ be a graph and ${\cal B} \subseteq 2^{V(G)}$ be a strict bramble of $G.$ Also let $X, Y \subseteq V(G)$ be covers of ${\cal B}$ and $S \subseteq V(G)$ be an $(X, Y)$-separator of $G.$ Then $S$ also covers ${\cal B}.$
\end{lemma}
\begin{proof}
Let $B \in {\cal B}.$ Recall that $B$ is connected and by definition, intersects both $X, Y.$ Since $S$ is an $(X, Y)$-separator of $G,$ $B$ also intersects $S.$
\end{proof}

\paragraph{Chordal graphs.}
A graph is {\em chordal} if all its induced cycles are triangles. The next proposition contains some 
folklore observations regarding separators in chordal graphs.

\begin{@menospreciando}\label{@circonlocution}
Let $G$ be a chordal graph and $x, y \in V(G).$ Any minimal $(x,y)$-separator of $G$ induces a clique. Given a minimal $(x,y)$-separator of $G$ there exist vertices $x'$ and $y'$ in the same connected components of $G - S$ as $x$ and $y$ respectively, such that $x'$ and $y'$ are adjacent to every vertex in $S.$ If  $S, S'$ are minimal separators of $G,$ then $S'$ is contained in some augmented connected component of $G - S$ and vice versa.
\end{@menospreciando}

\paragraph{Minors.}
Let $H, G$ be graphs. $H$ is a minor of $G$ and we write $H \leq_{\sf m} G$ if we can obtain $H$ from $G$ through a sequence of vertex deletions, edge deletions and edge contractions. A {\em minor-model of $H$ in $G$} is a function $μ : V(H) \to  2^{V(G)}$ that satisfies the following properties:
\begin{enumerate}
	\item for all $u \in V(H),$ $μ(u)$ is connected;
	\item for any pair $u, v \in V(H),$ $μ(u) \cap μ(v) = \emptyset$;
	\item if $\{u,v\} \in E(H)$ then there exists an edge $\{w,z\} \in E(G)$ such that $w \in μ(u)$ and $z \in μ(v).$
\end{enumerate}
It is well-known that $H$ is a minor of $G$ if and only if there exists a minor model of $H$ in $G.$


\paragraph{Tree decompositions.} 
\label{@imprecaciones}
Let $G$ be a graph, $T$ a tree and let $χ : V(T) \to 2^{V(G)}$ be a function mapping vertices of $T$ to subsets of vertices of $G.$ The pair $(T, χ)$ is a \textit{tree decomposition} of $G$ if it satisfies the following three conditions:
\begin{itemize}
    \item[(C1)] $\bigcup_{t \in V(T)} χ(t) = V(G)$;
    \item[(C2)] for every edge $e \in E(G),$ there is a vertex $t \in V(T),$ such that $e \subseteq χ(t)$;
    \item[(C3)] for every vertex $x \in V(G),$  the set $\{ t \mid x \in χ(t)\}$ is connected in $T.$
\end{itemize}

We refer to the vertices of $T$ as {\em the nodes} of $(T, χ)$ and to their images as {\em the bags} of $(T, χ).$ Given a vertex $x \in V(G),$ we define its {\em trace} in $(T,χ)$ as the set ${\sf Trace}_{(T,χ)}(x) = \{ t\in V(G) \mid x \in χ(t)\}.$ When $(T,χ)$ is clear from the context, we simply write ${\sf Trace}(x).$ Also for any leaf node $t \in V(T)$ with (the unique) incident edge $\{ t, t' \} \in E(T),$ we define the $t$-{\em petal} of $(T, χ),$ as ${\sf Petal}_{(T, χ)}(t) = χ(t) \setminus χ(t'),$ i.e. the private vertices of $t.$ When $(T,χ)$ is clear from the context, we simply write ${\sf Petal}(t).$ We define the \textit{width} of $(T, χ)$ as the maximum size of a bag of $(T, χ)$ minus one. The {\em treewidth} of a graph $G,$ denoted by $\tw(G),$ is the minimum width over all the tree decompositions of $G.$

%
%
%
\section{Lenient tree decompositions}
\label{@unpermissible}

In this section we provide some new concepts and we prove a series of preliminary results on lenient tree decompositions that will be useful for the proof of our main min-max equivalence Theorem in \autoref{@comprehension}.

Let $(T, χ)$ be a lenient tree decomposition of $G.$ The definitions of nodes, bags and $t$-petals of $(T, χ),$ as well as the trace of a vertex of $G,$  are identical to those we gave in the case of tree decompositions. We start with the following easy observation that follows directly from the definition of a lenient tree decomposition.

\begin{observation}\label{@intersubjective}
If $G$ and $H$ are two graphs where $H\leq_{\sf m}G$ and $G$ has a lenient tree decomposition of width at most $k,$ then so does $H.$  
\end{observation}

\begin{lemma}\label{@metaphysicians} Let $G$ be a graph and let $(T, χ)$ be a lenient tree decomposition of $G.$ Consider an internal node $t \in V(T)$ of $(T, χ)$ such that there exist different vertices $x,y \in V(G),$ such that, ${\sf Trace}(x)$ is a subtree of a subtree of $T - t$ and ${\sf Trace}(y)$ is a subtree of a different subtree of $T - t.$ Then $χ(t)$ is a $(x,y)$-separator of $G.$
\end{lemma}
\begin{proof} Fix a $x\mbox{-}y$ path $P$ in $G$ and consider the set $U = \bigcup_{u \in P} {\sf Trace}(u).$ Conditions \ref{@calamitously} and \ref{@unpretentious} imply that $U$ is connected in $T.$ So $t \in U,$ since $t$ is in the unique path of $T$ connecting the two subtrees. Thus there exists a vertex of $P$ intersecting $χ(t)$ which implies that it is a $(x,y)$-separator.
\end{proof}

\begin{lemma}\label{@nichtbestehens} Let $G$ be a graph and let $(T, χ)$ be a lenient tree decomposition of $G.$ Let $K \subseteq G$ be an induced clique in $G.$ Then there exists an adjacent pair of nodes $t, t' \in V(T)$ of $(T, χ)$ such that $V(K) \subseteq χ(t) \cup χ(t').$
\end{lemma}
\begin{proof} If there is a bag of $(T, χ)$ that contains $V(K)$ we are done. Assume otherwise. The proof proceeds by orienting every node $t \in V(T)$ towards the subtree containing $V(K) \setminus χ(t).$ By our assumption every leaf of $T$ points towards the inner part of $T.$ Consider an internal node $t \in V(T).$ By condition \ref{@unpretentious}, for every vertex in $x \in V(K) \setminus χ(t),$ ${\sf Trace}(x)$ is exclusively contained in some subtree of $T - t.$ Then, since $K$ is a clique, because of \autoref{@metaphysicians}, $\bigcup_{x \in V(K) \setminus χ(t)} {\sf Trace}(x)$ is exclusively contained in the same subtree of $T - t.$ This implies that $t$ can only point in a single direction and also that adjacent nodes of $(T, χ)$ cannot both point outward. Consider a maximal directed path that respects the directions given to the nodes of $(T, χ),$ which terminates at node $t \in V(T).$ Then there is a $t' \in V(T),$ such that $\{ t, t' \} \in E(T),$ pointing towards $t.$ Then $V(K) \subseteq χ(t) \cup χ(t').$
\end{proof}

\begin{lemma}\label{@deliberately} Let $G$ be a graph and let $(T, χ)$ be a lenient tree decomposition of $G$ such that $|V(T)| \geq 3.$ Also let $S \subseteq V(G)$ be connected. Then for any three nodes $t, t', t'' \in V(T)$ such that $t'$ is an internal node of the unique $t\text{-}t''$ path $P$ in $T,$ if $χ(t) \cap S \neq \emptyset$ and $χ(t'') \cap S \neq \emptyset$ then $χ(t') \cap S \neq \emptyset.$
\end{lemma}
\begin{proof}
Let $x \in χ(t) \cap S,$ $y \in χ(t'') \cap S.$ If $x = y$ then by condition \ref{@unpretentious}, $x \in χ(t').$ Else if $x$ and $y$ are adjacent then by condition \ref{@calamitously}, there exists a pair of close nodes $u, v \in V(T)$ where $\{ x, y\} \subseteq χ(u) \cup χ(v).$ By condition \ref{@unpretentious}, we can assume that $u, v \in V(P)$ and thus $P \subseteq {\sf Trace}(x) \cup {\sf Trace}(y),$ which implies that either $x \in χ(t')$ or $y \in χ(t').$ Else if $x$ and $y$ are not adjacent, then, by \autoref{@metaphysicians}, $χ(t')$ is an $(x,y)$-separator, and since $S$ is connected it must intersect $χ(t').$
\end{proof}

\paragraph{Extreme lenient tree decompositions.}
We call a lenient tree decomposition $(T, χ)$ of width $k$ {\em extreme}, if it satisfies the following properties.
\begin{itemize}
\item All bags of $(T, χ)$ are of equal size. 
\item No bag   of $(T, χ)$ is a subset of another one. 
\item For any node $t \in V(T),$ such that ${\sf deg}_{T}(t) = 2,$ and for any two nodes $t', t'' \in V(T)$ such that $t \in V(P),$ where $P$ is the unique $t'\text{-}t''$ path in $T,$ it holds that $χ(t) \nsubseteq χ(t') \cup χ(t'').$ 
\item For every pair $t', t'' \in V(T)$ of different leaf neighbours of a node $t \in V(T),$ $|{\sf Petal}(t') \cup {\sf Petal}(t'')| > k.$
\end{itemize}

\begin{lemma}\label{@comprehensibility} Let $G$ be a graph. If there exists a lenient tree decomposition of $G$ of width $k,$ then there exists an extreme lenient tree decomposition $(T, χ)$ of $G$ of the same width.
\end{lemma}
\begin{proof} Define $(T, χ)$ as the lenient tree decomposition of width $k,$ that minimizes:
\begin{eqnarray}
k|V(T)| - \sum_{t \in V(T)} |χ(t)|\label{@ridiculement}
\end{eqnarray}
and subject to \eqref{@ridiculement}, minimizes:
\begin{eqnarray}
|V(T)|.\label{@gegenstandes}
\end{eqnarray}
 We prove that $(T, χ)$ is extreme. If $|V(T)| \leq 1,$ then our claim trivially holds. Now, assume it is not. We distinguish cases.

Suppose that there exists a pair of bags of $(T, χ)$ with unequal size. Then there is also a pair of adjacent nodes $t, t' \in V(T)$ such that $|χ(t)| < |χ(t')|.$ Then we can add vertices of $χ(t') \setminus χ(t)$ to $χ(t).$ This contradicts the minimality of \eqref{@ridiculement}.

Now, suppose that there exists a bag that is a subset of another bag of $(T, χ).$ Then, because of condition \ref{@unpretentious}, there is also a pair of adjacent nodes $t, t' \in V(T)$ such that $χ(t) \subseteq χ(t').$ Then we can remove node $t$ from $(T, χ)$ and connect its neighbours with $t'.$  This contradicts the minimality of \eqref{@ridiculement}.

Let $t \in V(T)$ with ${\sf deg}_{T}(t) = 2$ and a pair of nodes $t', t'' \in V(T),$ such that $t \in V(P),$ where $P$ is the unique $t'\text{-}t''$ path in $T.$ Suppose that $χ(t) \subseteq χ(t') \cup χ(t'').$ Then we can remove $t,$ and make $t',t''$ adjacent.  This contradicts the minimality of \eqref{@gegenstandes}.

Finally, suppose that there is a node $t \in V(T),$ such that there exists a pair $t', t'' \in V(T)$ of different leaf neighbours of $t$ such that $|{\sf Petal}(t') \cup {\sf Petal}(t'')| \leq k.$ Then we can identify the two leaves into a single leaf, whose bag will contain ${\sf Petal}(t') \cup {\sf Petal}(t''),$ plus some additional vertices of $χ(t) \setminus ({\sf Petal}(t') \cup {\sf Petal}(t'')),$ to ensure that the bag corresponding to the new leaf has size $|χ(t)|.$ This contradicts the minimality of \eqref{@gegenstandes}.
%
\end{proof}

Given a graph $G$ and a lenient tree decomposition $(T, χ)$ of $G,$ we define the {\em $(T, χ)$-completion} of $G$ as the graph $G^{+} = (V(G),E^+)$ where,
\begin{align*}
&E^+=\bigcup_{{t,t'\in V(T):}\atop{\text{$t,t'$ are close in $T$}}} {χ(t) \cup χ(t^{\prime})\choose 2}.
\end{align*}
That is, we add edges (if they do not already exist) between all vertices of each bag or of each two ``neighboring'' bags. 
Clearly, $(T, χ)$ is also a lenient tree decomposition of $G^{+}.$

\begin{lemma}\label{@ridiculously} Let $G$ be a graph and let $(T, χ)$ be a lenient tree decomposition of $G.$ Also let $G^{+}$ be the $(T, χ)$-completion of $G.$ Then $G^{+}$ is chordal.
\end{lemma}
\begin{proof}
Suppose that there exists a pair $x,y \in V(G)$ of non-adjacent vertices of $G^{+}$ that belong to an induced cycle of size at least four. Since $x, y$ are not adjacent and the union of the bags of any adjacent pair of nodes of $(T, χ)$ induces a clique in $G^{+},$ the closest nodes $t, t' \in V(T),$ whose bags contain $x,y$ respectively are at distance at least two in $T.$ Let $t'' \in V(T)$ be an internal vertex of a $t\text{-}t'$ path in $T.$ Because of \autoref{@metaphysicians}, $χ(t'')$ is a $(x,y)$-separator in $G^{+}.$ Let $P, P' \subseteq G^{+}$ be the two internally vertex disjoint $x\text{-}y$ paths that define this cycle. These paths intersect $χ(t'')$ and since $χ(t'')$ is a clique, the cycle has two non-consecutive vertices that are adjacent, which contradicts our assumption. Thus $G^{+}$ is chordal.
\end{proof}

\begin{lemma}\label{@saupoudrerez} Let $G$ be a graph and let $(T, χ)$ be an extreme lenient tree decomposition of $G$ of width $k.$ Also let $G^{+}$ be the $(T, χ)$-completion of $G.$ Then for any pair of nodes $t, t' \in V(T)$ of $(T, χ)$ there exist $k$ disjoint $χ(t)\text{-}χ(t')$ paths in $G^{+}.$
\end{lemma}
\begin{proof}
The proof proceeds by induction on the distance of $t, t'$ in $T.$ If $t = t'$ or $t, t'$ are adjacent, the claim trivially holds. Assume that $t, t'$ are at distance $r > 1.$ Consider the unique $t\text{-}t'$ path in $T$ and consider the neighbour of $t'$ on this path, say $t''.$ By the inductive hypothesis there are $k$ disjoint $χ(t)\text{-}χ(t'')$ paths in $G.$ Since $χ(t') \cup χ(t'')$ induces a clique in $G^{+}$ we can easily extend these paths to disjoint $χ(t)\text{-}χ(t')$ paths.
\end{proof}

\begin{lemma}\label{@manipulationen} Let $G$ be a graph and let $(T, χ)$ be an extreme lenient tree decomposition of $G$ of width $k.$ Also let $G^{+}$ be the $(T, χ)$-completion of $G.$ Then, if $S \subseteq V(G^{+})$ is a minimal $(x,y)$-separator of $G^{+},$ there exists a node $t \in V(T)$ of $(T, χ)$ such that $S = χ(t).$
\end{lemma}
\begin{proof}
Since $x, y \in V(G^{+})$ are not adjacent they do not belong in an adjacent pair of bags of $(T, χ).$ Then \autoref{@saupoudrerez} easily implies that there are $k$ internally vertex disjoint $x\text{-}y$ paths in $G^{+}.$ Then, by Menger's Theorem, $S$ has size at least $k.$ Additionally, by \autoref{@ridiculously}, $G^{+}$ is chordal and thus $S$ induces a clique in $G^{+}.$ Then, \autoref{@nichtbestehens}, implies that there exists an adjacent pair of nodes $t, t' \in V(T)$ such that $S \subseteq χ(t) \cup χ(t').$ Also, since the closest bags of $(T, χ)$ that contain $x,y$ are at distance at least two, at least one of $x$ or $y$ {cannot be in  $χ(t)\cup χ(t')$}. Assume it is $x$ and that $χ(t)$ is closer than $χ(t')$ to the closest bag containing $x.$ Since $G^{+}$ is chordal we can assume that $x$ is adjacent to all vertices of $S.$ Now observe that conditions \ref{@calamitously} and \ref{@unpretentious} imply that $x$ cannot be adjacent with any vertex in $χ(t') \setminus χ(t)$ which implies that $S$ cannot contain any vertex in $χ(t') \setminus χ(t).$ Thus $S = χ(t).$
\end{proof}

\begin{lemma}\label{@perspiration} Let $G$ be a graph and let $(T, χ)$ be an extreme lenient tree decomposition of $G.$ Also let $G^{+}$ be the $(T, χ)$-completion of $G.$ Then there is a unique bijection between minimal separators $S \subseteq V(G^{+})$ of $G^{+}$ and nodes $t \in V(T)$ of $(T, χ),$ such that $S = χ(t),$ ${\sf cdeg}_{G^{+}}(S) = {\sf deg}_{T}(t)$ and for any maximal clique $K \subseteq G^{+},$ such that $S \subseteq V(K),$ there is a node $t' \in V(T),$ adjacent to $t,$ such that $V(K) = χ(t) \cup χ(t').$
\end{lemma}
\begin{proof} Because of \autoref{@manipulationen}, for every minimal separator $S \subseteq V(G^{+})$ of $G^{+},$ there is a node $t \in V(T),$ such that $S = χ(t),$ where clearly $t$ is an internal node of $T.$ Now, observe that, since every pair of different bags of $(T, χ)$ is not a subset of one another, for any internal node $t \in V(T),$ there exists a pair of vertices whose trace in $T$ belongs in different subtrees of $T - t.$ Then, because of \autoref{@metaphysicians}, every internal bag of $(T, χ)$ is a separator of $G^{+},$ and because of \autoref{@manipulationen}, it also has to be minimal. Also, since no two bags of $(T, χ)$ are equal, this bijection is unique. Then, let $S \subseteq V(G^{+})$ be a minimal separator of $G^{+}$ and $t \in V(T)$ be the unique internal node of $T$ such that $S = χ(t).$ Because of \autoref{@nichtbestehens}, for any maximal clique $K \subseteq G^{+},$ such that $S \subseteq V(K),$ there must be a node $t' \in V(T),$ adjacent to $t,$ such that $V(K) = χ(t) \cup χ(t').$ Additionally, ${\sf deg}_{T}(t) \leq {\sf cdeg}_{G^{+}}(S),$ since each connected component of $G^{+} - S$ is contained in some subtree of $T - t.$ Also ${\sf cdeg}_{G^{+}}(S) \leq {\sf deg}_{T}(t),$ since each of these subtrees must induce a connected graph since $G^{+}$ is the $(T, χ)$-completion of $G.$
\end{proof}

\paragraph{Amalgamations of lenient tree decompositions.}
Our min-max theorem (\autoref{@comprehension}) is using the technique of  Bellenbaum and Diestel~\cite{BellenbaumD02twosh} for proving the equivalence between the bramble number (that is a max-min parameter) and treewidth (that is a min-max parameter).  An important ingredient of the proof of Bellenbaum and Diestel~\cite{BellenbaumD02twosh} is the  concept of  {\sl amalgamating tree decompositions}. We next adapt it to lenient tree decompositions.

Let $(T, χ)$ be a lenient tree decomposition of a graph $G.$ Let $S \subseteq V(G),$ $C \in {\sf Acc}(G, S)$ and $C^{-} = C - S.$ Let $s \in V(T)$ be a node of $(T, χ)$ and for every $x \in S,$ let $Z_{x} \subseteq T$ be an $s\text{-}{\sf Trace}(x)$ path in $T.$ Then we define $(T,χ')$ so that for every  $t \in V(T)$ we set,
\begin{align*}
χ'(t) =  (χ(t) \cap V(C^{-})) \cup \{ x \in S \mid t \in V(Z_{x})\}.
\end{align*}

It is easy to observe that $(T, χ')$ is a lenient tree decomposition of $C,$ where we force $S\subseteq χ'(s)$ while we fix condition \ref{@unpretentious} of the definition with the necessary addition of vertices to every bag of $(T, χ).$ We say that $(T, χ')$ is an {\em amalgamated restriction of $(T, χ)$ on $C$ with respect to $s$}.

Let $G$ be a graph and $S \subseteq V(G).$ Let ${\cal T} = \{ (T_{C}, χ_{C}) \mid C \in {\sf Acc}(G, S))\}$ be a family, where for each $C \in {\sf Acc}(G, S),$ $(T_{C}, χ_{C})$ is a lenient tree decomposition of $C$ with a node $s_{C} \in V(T_{C})$ where $S \subseteq χ_{C}(s_{C}).$ We build from ${\cal T},$ a lenient tree decomposition $(T, χ)$ of $G$ as follows. $T$ is obtained by $\bigcup_{C \in {\sf Acc}(G,C)}T_{C}$ after adding a new node $t_{\rm new}$ and for every $(T_{C}, χ_{C}) \in {\cal T},$ making $t_{\rm new}$ adjacent with $s_{C}.$ We finally define,
\begin{align*}
&χ'=\{(t_{\rm new},S)\}\cup \bigcup_{C\in{\sf Acc}(G,C)}χ_{C}.
\end{align*}
Observe that $(T, χ)$ is a lenient tree decomposition of $G.$ We call it the $S$-{\em amalgamation} of ${\cal T}.$ By the above construction we observe the following.

\begin{observation}\label{@desapasionado}
Let $k \in {\Bbb N}.$ If every $(T_{C}, χ_{C}) \in {\cal T}$ has width at most $k,$ then the $S$-amalgamation of ${\cal T}$ also has width at most $k.$
\end{observation}

Following the ideas of~\cite{BellenbaumD02twosh}, we prove the following lemma.

\begin{lemma}\label{@dilaceraciones} Let $G$ be a graph, $S\subseteq V(G)$ and let $C\in{\sf Acc}(G,S).$ Also let $(T, χ)$ be a lenient tree decomposition of $G,$ $s\in V(T),$ and $(T',χ')$ be the amalgamated restriction of $(T,χ)$ on $C$ with respect to $s.$ Also, suppose that $G - (V(C) \setminus S)$ contains a set $\{ P_{x} \mid x \in S\}$ of disjoint $S\text{-}χ(s)$ paths where $x$ is a terminal vertex of $P_{x}.$ Then for every node $t\in V(T),$ $|χ'(t)| \leq |χ(t)|.$
\end{lemma}
\begin{proof}
Let $t \in V(T)$ be a node such that there is an $x \in χ'(t) \setminus χ(t).$ By definition of $χ'(t),$ we have that $x \in S.$ Let $Z_{x} \subseteq T$ be an $s\text{-}{\sf Trace}(x)$ path in $T.$ By definition of $χ'(t),$ we have that $t \in V(Z_{x}).$ Since $x \notin χ(t),$ then $t \notin {\sf Trace}(x),$ which implies that  $Z_{x}$ has length at least $1.$ Moreover, if $t = s,$ then trivially $χ(s)$ contains the terminal of $P_{x}$ that belongs in $χ(s).$ Thus, we can assume that $t$ is an internal node of $T.$ Let $y$ be the terminal vertex of $P_{x}$ in $χ(s).$ Note that, if $P_{x}$ has length $1,$ then $x, y$ are adjacent and by condition \ref{@calamitously}, there exists a pair of close nodes $u, v \in V(T)$ where $\{ x, y\} \subseteq χ(u) \cup χ(v).$ By condition \ref{@unpretentious}, we can assume that $u, v \in V(Z_{x})$ and thus $Z_{x} \subseteq {\sf Trace}(x) \cup {\sf Trace}(y),$ which, since $x \notin χ(t),$ implies that $y \in χ(t).$ Now, assume that, $P_{x}$ has length at least $2$ and that $x, y$ are not adjacent. Then, because of \autoref{@metaphysicians}, $χ(t)$ is an $(x,y)$-separator, and hence contains some other distinct vertex of $P_{x}.$ Note that under the assumption that $x \in χ'(t) \setminus χ(t),$ any such vertex is not contained in $χ'(t),$ since $χ'(t) \subseteq V(C)$ while $V(P_{x}) \setminus \{ x \} \subseteq V(G) \setminus V(C).$ Thus $|χ'(t)| \leq |χ(t)|.$
\end{proof}


\section{$k$-domino-trees}
\label{@monosyllabic}

Recall that $k$-trees serve as the edge-extremal graphs of graphs of bounded bramble number (via the equivalence with treewidth). In this section we define the concept of a {\em $k$-domino-tree} that 
is the corresponding extremal structure for the strict brambles. As we will see, this notion is more entangled than $k$-trees. 

Let $G$ be a chordal graph. We call a maximal clique of $G,$ {\em external} (respectively {\em internal}), if its vertex set contains at most one (respectively at least two) minimal separator(s) of $G.$ We say that all external maximal cliques containing the same minimal separator $S,$ form an {\em external family of} $S,$ and we denote it by ${\cal K}_{G}(S).$ For each   $K \in {\cal K}_{G}(S),$ we define its {\em valiancy} to be ${\sf val}(K) = |V(K) \setminus S|,$ i.e. the number of private vertices of $K.$ We call a minimal separator $S,$ {\em external} (respectively {\em internal}), if ${\cal K}_{G}(S) \neq \emptyset$ (respectively ${\cal K}_{G}(S)= \emptyset$).

\smallskip
\noindent
Let $k \in \mathbb{N}.$ A graph $G$ is a $k$-{\em domino-tree} if it is either $K_{r}$ for some $r\leq k,$ or it satisfies the following properties:
\begin{enumerate}[label={\it \roman*}.]
	\setlength\itemsep{-2px}
	\item\label{@unswervingly} $G$ is chordal;
	\item\label{@kennzeichnet} Every minimal separator of $G$ has size $k$;
	\item\label{@correspondiente} Every maximal clique of $G$ has size in  $[k+1, 2k]$;
	\item\label{@begriffswort} The vertex set of every maximal clique of $G$ contains at most two minimal separators;
	\item\label{@precapitalist} The vertex set of every maximal clique of $G$ that contains exactly two minimal separators $S, S'$ is equal to $S \cup S'$;
	\item\label{@mandamientos} Every internal minimal separator of $G$ of connectivity-degree two, is not contained in the union of two other minimal separators;
	\item\label{@presentiment} For every external minimal separator $S$ of connectivity-degree  two, the union of the vertex sets of the  maximal cliques that contain $S,$ has size greater than $2k$;
	\item\label{@bezeichnenden} For every external minimal separator $S,$ with $|{\cal K}_{G}(S)| > 1,$ for any different pair $K, K' \in {\cal K}_{S},$ ${\sf val}(K) + {\sf val}(K') > k.$
\end{enumerate}
A graph $G$ is a {\em partial $k$-domino-tree} if it is  a spanning subgraph of a $k$-domino-tree.
\smallskip

\begin{figure}[t]
\begin{center}
\graphicspath{{./Figures/}}
\scalebox{1}{\includegraphics{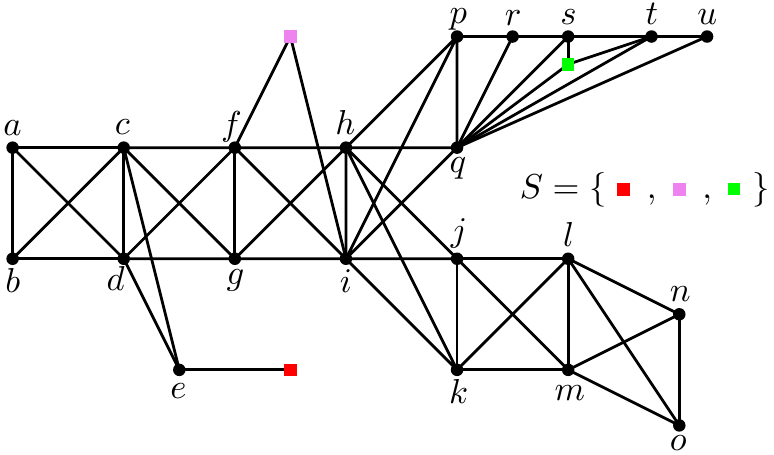}}
\end{center}
\caption{A graph $G$ with a set $S$ of three square vertices, such that $G'=G\setminus S$ is  a $2$-domino-tree. The neighbor of the \red{red} square vertex is a cut-vertex, which violates property \ref{@kennzeichnet}. Also, the $K_{2}$ incident to the \red{red} square, violates property \ref{@correspondiente}. The \violet{violet} square vertex cannot exist, as otherwise the $K_{3}$ induced by it and its neighbors is separated by $\{i,f\},$ which is contained in the maximal clique induced by $\{f,g,h,i\}$ which contains three minimal separators. This violates property \ref{@begriffswort}. The maximal clique induced by the \green{green} square vertex and its neighbors contains two minimal separators that do not cover the \green{green} square vertex. This violates property \ref{@precapitalist}.  Consider now the graph $G'.$ If we remove from $G'$ the edge $\{c,d\},$ the graph will no longer be chordal, which violates property \ref{@unswervingly}. If we remove the {edge $\{k,l\}$}, the set $\{j,m\}$ becomes an internal minimal separator of connectivity-degree two, that is contained in the union of the minimal separators $\{j,k\}$ and $\{l,m\},$ which violates property \ref{@mandamientos}. If we remove the edge $\{l,o\},$ the set $\{m,n\}$ becomes an external minimal separator of connectivity-degree two, where the union of the vertex sets of the maximal cliques induced by $\{l,m,n\}$ and $\{m,n,o\},$ has size four, which violates \ref{@presentiment}. If we remove edge $\{a,b\},$ the external minimal separator $\{c,d\}$ belongs to the external maximal cliques induced by $\{a,c,d\}$ and $\{b,c,d\}$ and the sum of the valiances of these two cliques is two, which violates property \ref{@bezeichnenden}.}
\label{domino_expl}
\end{figure}

See \autoref{domino_expl} for an example of the above definition. We proceed with a few remarks. By \ref{@unswervingly}, in a $k$-domino-tree, every minimal separator $S$ is contained in the vertex set of a maximal clique. Moreover, it is easy to observe that, each augmented connected component in ${\sf Acc}(G, S),$ contains exactly one maximal clique whose vertex set contains $S.$ Therefore, a minimal separator of connectivity-degree $d$ is contained in exactly $d$ different maximal cliques. Also for any $K \in {\cal K}_{G}(S),$ ${\sf val}(K) \in [1, k].$ Moreover, for  different pairs  $K, K' \in {\cal K}_{G}(S),$ $V(K) \cap V(K') = S,$ so ${\sf val}(K) + {\sf val}(K')$ does not double count vertices. Also, if $S$ is external, the connectivity-degree of $S$ is at least $|{\cal K}_{G}(S)|.$

The following Lemma will be very useful in our proof of the min-max theorem in \autoref{@comprehension}.
\begin{lemma}\label{@manifestations} Let $G$ be a $k$-domino-tree and $S \subseteq V(G)$ be a minimal separator of $G$ such that there exists $C \in {\sf cc}(G - S)$ with $|V(C)| > k.$ Then there exists a minimal separator $S' \subseteq V(G)$ of $G$ with the following properties:
\begin{enumerate}
	\setlength\itemsep{-2px}
	\item $S'$ is properly contained in $C^{+},$ where $C^{+} = G[V(C) \cup S],$ i.e. $C^{+}$ is the augmented connected component of $G - S,$ corresponding to $C.$
	\item $S$ and $S'$ are not a subset of one another.
	\item The vertices in $S \cup S'$ induce a maximal clique in $G.$
	\item $C - S'$ is a connected component of $G - S'.$
\end{enumerate}
\end{lemma}
\begin{proof} Since every maximal clique of $G$ has size at most $2k$ and $|V(C)| > k,$ there exists a pair of non adjacent vertices $x \in S,$ $y \in V(C),$ in $G.$ Let $S' \subseteq V(G)$ be a minimal $(x, y)$-separator maximizing the size of $C'_{y} \in {\sf Acc}(G, S'),$ where $C'_{y}$ is the augmented connected component of $G - S',$ such that $y \in V(C'_{y}).$

Property {\em 1.} holds, since there exists a $x\text{-}y$ path with internal vertices in $C,$ which implies that there is a vertex of $C$ in $S',$ which, since $G$ is chordal (property \ref{@unswervingly}), implies that $S' \subseteq V(C^{+}).$

Property {\em 2.} also holds since $x \in S,$ while $ x \notin S'$ and $S'$ contains a vertex of $C$ which clearly cannot be in $S.$


For property {\em 3.} assume that $S \cup S'$ does not induce a clique in $G.$ Then there exists a pair of non adjacent vertices $x' \in S,$ $y' \in S'$ which in turn implies the existence of a minimal $(x', y')$- separator $S'' \subseteq V(G)$ of $G,$ which is also properly contained in $C^{+}.$ Observe that since $x'$ and $y'$ are not adjacent, $x' \notin S',$ which since $S$ is a clique, implies that $S'$ is also an $(x',y)$-separator. Now let $C'_{x'} \in {\sf Acc}(G, S'),$ such that $x' \in V(C'_{x'})$ and $C''_{y} \in {\sf Acc}(G, S''),$ such that $y \in V(C''_{y}).$ Notice that every $x'\text{-}y'$ path has an internal vertex in $C'_{x'}$ and it intersects $S''.$ This implies that $S''$ is completely contained in $C'_{x'}$ and in turn that $y \notin S''.$ Finally observe that there exists a $y\mbox{-}y'$ path with internal vertices only in $C'_{y}$ and $S''$ cannot intersect any such path which implies that $y' \in V(C''_{y}).$ This contradicts the assumption on $S'.$ Also since $S \cup S'$ induces a clique in $G,$ it is contained in a maximal clique of $G$ whose vertex set contains exactly two minimal separators, $S$ and $S'.$ Then $S \cup S'$ induces the entire maximal clique (property \ref{@precapitalist}).


For property {\em 4.}, let $C'_{x} \in {\sf Acc}(G, S'),$ such that $x \in V(C'_{x}).$ It suffices to prove that $(V(C'_{x}) \setminus S') \cap V(C)$ is empty. Then $C'_{y} - S'$ is the required connected component. Assume that there exists a vertex $z \in (V(C'_{x}) \setminus S') \cap V(C).$ Since $S \cup S'$ induces a maximal clique in $G,$ there exists a vertex $w \in S \cup S',$ such that $\{ w, z \} \notin E(G).$ Let $S''$ be a minimal $(w, z)$-separator in $G,$ such that $S'' \subseteq S \cup S'.$ Clearly such a separator exists. But then $S \cup S'$ contains at least three minimal separators which contradicts property \ref{@begriffswort}.
\end{proof}

\begin{figure}[t]
\begin{center}
\graphicspath{{./Figures/}}
\scalebox{0.8}{\includegraphics{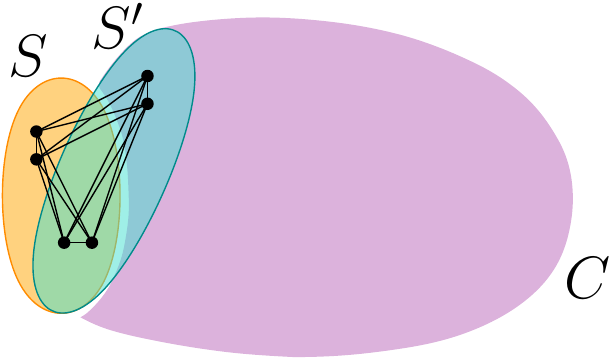}}
\end{center}
\caption{The connected component $C,$ and the minimal separators $S$ and $S'$ of \autoref{@manifestations}.}
\label{@unrecognized}
\end{figure}

\section{The min-max equivalence Theorem}
\label{@comprehension}

In this section we prove the main result of this paper. We prove equivalent min-max formalizations for the strict bramble number in terms of the lexicographic tree product number, the minimum width of lenient tree decompositions and subgraph containment in extremal structures. Our main result is the following.



\begin{theorem}\label{@interrelated}
Let $G$ be a graph and $k\in\Bbb{N}.$ The following statements are equivalent.  
\begin{enumerate}
	\setlength\itemsep{-2px}
	\item There is a tree $T$ such that $G$ is a minor of $T\cdot K_{k}.$
	\item $G$ has a lenient tree decomposition of width at most $k.$
	\item $G$ has no strict bramble of order greater than $k.$
	\item  $G$ is a partial $k$-domino-tree.
\end{enumerate}
\end{theorem}

\begin{proof} ({\em 2}$\ \Rightarrow\ ${\em 1}\,). Let $(T, χ)$ be a lenient tree decomposition of $G$ of width $k.$ Because of \autoref{@comprehensibility}, we may assume that $(T, χ)$ is also extreme. Let $G^{+} = T \cdot K_{k}.$ It is enough to prove that $G \leq_{\sf m} G^{+}.$ Let $μ: V(G)\to V(G^+)$ be a mapping such that for each $v\in V(G),$ $μ(v)$ contains exactly one vertex from each clique corresponding to the vertices of ${\sf Trace}(v).$ Observe that since $(T, χ)$ is extreme, these vertices can be selected such that, for different $v,v'\in V(G),$ $μ(v)$ and $μ(v')$ are disjoint. Also for adjacent vertices $u, v \in V(G)$ it is easy to see that they contain a pair of vertices $u', v' \in V(G^{+}),$ from close nodes $t \in {\sf Trace}(u)$ and $t' \in {\sf Trace}(v),$ which are adjacent in $G^{+}.$ In any case $μ$ is a minor model of $G$ in $G^+,$ as required.


({\em 1}$\ \Rightarrow\ ${\em 2}\,). Let $(T,χ)$ be such that $χ$ maps each node of $Τ$ to the corresponding clique in $T \cdot K_{k}.$ Clearly $(T,χ)$ is a lenient tree decomposition of $T \cdot K_{k}$ and the result follows from \autoref{@intersubjective}.


({\em 4}$\ \Rightarrow\ ${\em 3}\,). The claim trivially holds if $G$ has at most $k$ vertices. Let $D$ be a $k$-domino-tree such that $G$ is a spanning subgraph of $D.$ Assume to the contrary that $G$ has a strict bramble ${\cal B} \subseteq 2^{V(G)}$ of order greater than $k.$ ${\cal B}$ is also a strict bramble of the same order in $D.$ Every minimal separator $S \subseteq V(D)$ of $D$ has size $k$ and, as such, there exists an element of ${\cal B}$ which it does not intersect. This element is connected and thus must be completely contained in some connected component $C \in {\sf cc}(D - S).$ Choose $S$ with the smallest such component. If $C$ has size $k$ we are done. Otherwise observe that there exists a separator $S' \subseteq V(D)$ of $D$ with the properties stated at \autoref{@manifestations}. With the same reasoning there exists a connected component $C' \in {\sf cc}(D - S')$ that contains an element of $\cal B.$ Properties {\em 1}, {\em 2} and {\em 4} of \autoref{@manifestations} certify that there exists a connected component of $D - S'$ properly contained in $C.$ $C'$ cannot be this component as this would contradict the choice of $S.$ Property {\em 4} also certifies that the vertex set of any other connected component of $D - S'$ is disjoint from the vertex set of $C.$ This contradicts the intersecting properties of ${\cal B}.$

({\em 2}$\ \Rightarrow\ ${\em 4}\,). The claim trivially holds if $G$ has at most $k$ vertices. Let $(T, χ)$ be a lenient tree decomposition of $G$ of width $k.$ As before, because of \autoref{@comprehensibility}, we may assume that $(T, χ)$ is extreme. Let $D$ be the $(T, χ)$-completion of $G.$ Clearly $D$ is a spanning supergraph of $G.$ We argue that $D$ is a $k$-domino-tree. {\em Property \ref{@unswervingly}} is obtained directly from \autoref{@ridiculously}. {\em Property \ref{@kennzeichnet}} is implied from \autoref{@manipulationen} since all bags of $(T, χ)$ have size $k.$ For {\em Property \ref{@correspondiente}}, because of \autoref{@nichtbestehens}, for every maximal clique $K$ of $D,$ there is an adjacent pair of nodes $t, t' \in V(T)$ such that $V(K) \subseteq χ(t) \cup χ(t').$ This immediately implies that $|V(K)| \leq 2k.$ Also observe that, since $χ(t)$ and $χ(t')$ are not a subset of one another, $|V(K)| \geq k+1.$ {\em Property \ref{@begriffswort}} is implied from \autoref{@saupoudrerez}, since if $V(K)$ contains a minimal separator of $D,$ then it is either $χ(t)$ or $χ(t').$ Also, {\em Property \ref{@precapitalist}} holds since, $χ(t) \cup χ(t') = V(K).$ Now, consider a minimal separator $S \subseteq V(D)$ of ${\sf cdeg}_{D}(S) = 2.$ Then, by \autoref{@perspiration}, let $t \in V(T)$ be the corresponding node of ${\sf deg}_{T}(t) = 2.$ Since $(T, χ)$ is extreme, for any pair of nodes $t', t'' \in V(T)$ such that $t \in V(P),$ where $P$ is the unique $t'\text{-}t''$ path in $T,$ we have that $χ(t) \nsubseteq χ(t') \cup χ(t'').$ This easily implies {\em property \ref{@mandamientos}} and {\em property \ref{@presentiment}} Additionally, if $S$ is external, $t$ has exactly $|{\cal K}_{G}(S)|$ leaf neighbours in $T.$ Moreover, if $|{\cal K}_{G}(S)| > 1,$ let $t', t'' \in V(T)$ be two different leaf neighbours of $t.$ We know that, $|{\sf Petal}(t') \cup {\sf Petal}(t'')| > k.$ Now observe that, because of \autoref{@perspiration}, $χ(t) \cup χ(t')$ and $χ(t) \cup χ(t'')$ correspond to two maximal cliques $K, K' \in {\cal K}_{G}(S)$ and that ${\sf Petal}(t') = {\sf val}(K)$ while ${\sf Petal}(t'') = {\sf val}(K').$ Thus {\em property \ref{@bezeichnenden}} is also satisfied.

{\em 3}$\ \Rightarrow\ ${\em 2}\,). This part of the proof uses the ideas of  the corresponding proof in~\cite{BellenbaumD02twosh}.
Assume that $G$ has no strict bramble of order greater than $k.$ We show that for every strict bramble ${\cal B} \subseteq 2^{V(G)}$ of $G,$ there is a lenient tree decomposition of $G$ such that, if the size of a bag 
is greater than $k$ then this bag does not cover $\mathcal{B}.$ We call such a lenient tree decomposition ${\cal B}$-{\em admissible}. Observe that for $\mathcal{S} = \emptyset,$ the width of a ${\cal B}$-admissible lenient tree decomposition of $G$ is at most $k,$ since trivially any set covers an empty strict bramble. To prove that $G$ admits a ${\cal B}$-admissible lenient tree decomposition, we will prove instead that every $C \in {\sf Acc}(G, S),$ admits a ${\cal B}$-admissible lenient tree decomposition $(T_{C}, χ_{C}),$ with a node $s \in V(T_{C}),$ such that $χ_{C}(s) = S.$ Then the $S$-amalgamation of $\{ (T_{C}, χ_{C}) \mid C \in {\sf Acc}(G, S)\}$ is clearly ${\cal B}$-admissible.

Let ${\cal B} \subseteq 2^{V(G)}$ be a strict bramble and let $S \subseteq V(G)$ be a minimal cover of ${\cal B}.$ Then ${\sf order}({\cal B}) = |S| \leq k.$ If $S = V(G)$ then the lenient tree decomposition with $S$ as the only bag, is ${\cal B}$-admissible. Assume otherwise. We proceed with a backwards induction on the size $|\mathcal{B}|.$

For the base case assume that any ${\cal B'} \subseteq 2^{V(G)}$ such that $|{\cal B'}| > |{\cal B}|$ is not a strict bramble, i.e., ${\cal B}$ is a strict bramble with a maximum number of elements. We can safely assume this since $|{\cal B}| < 2^{V(G)}.$ This holds since ${\cal B}$ cannot contain a set and its complement. Let $C \in {\sf cc}(G - S),$ $C^{+} = G[V(C) \cup S]$ and let ${\cal B'} = {\cal B} \cup \{ C \}.$ Since ${\cal B'}$ is not a strict bramble it is implied that $C$ does not cover $\cal B.$ Then the pair $(T, χ)$ where $T$ consists of two adjacent nodes $t, t' \in V(T)$ while $χ(t) = S$ and $χ(t') = C$ is a lenient tree decomposition of $C^{+}$ with the required properties.

Now assume inductively that for any strict bramble ${\cal B'} \subseteq 2^{V(G)}$ such that $|{\cal B'}| > |{\cal B}|,$ there is a ${\cal B'}$-admissible lenient tree decomposition of $G.$ Let $C \in {\sf cc}(G - S),$ $C^{+} = G[V(C) \cup S]$ and let ${\cal B'} = {\cal B} \cup \{ C \}.$ If ${\cal B'}$ is not a strict bramble then we proceed as in the base case. Assume that it is. Then, since $S$ covers ${\cal B}$ and $S \cap C = \emptyset,$ we have that $C \notin {\cal B}$ and thus $|{\cal B'}| > |{\cal B}|.$ Then by the induction hypothesis, there is a ${\cal B'}$-admissible lenient tree decomposition of $G,$ say $(T, χ).$ Ιf $(T, χ)$ is also $\cal{B}$-admissible then we are done. Otherwise, there exists a node $s \in V(T)$ of $(T, χ),$ such that $|χ(s)| > k,$ that covers ${\cal B}$ but does not cover ${\cal B'}.$ Since it does not cover ${\cal B'},$ it is implied that $χ(s) \subseteq V(G - C).$ Moreover, because of \autoref{@constellations}, any $(S,χ(s))$-separator of $G$ has size at least ${\sf order}({\cal B}).$ By Menger's Theorem, there exists a set $\{ P_{x} \mid x \in S \}$ of disjoint $S\text{-}χ(s)$ paths in $G.$ Since ${\sf order}({\cal B}) = |S|$ and $|S| < |χ(s)|,$ we can also assume that for every $x \in S,$ $x$ is the endpoint of $P_{x}$ in $S.$ Finally, observe that since $χ(s) \subseteq V(G - C),$ we can also assume that for each $x \in S,$ $V(P_{x}) \subseteq V(G - C).$

Now let $(T', χ')$ be the amalgamated restriction of $(T, χ)$ on $C^{+}$ with respect to $s.$ First observe that since $χ(s) \subseteq V(G - C),$ we have that $χ'(s) = S.$ It remains to show that $(T', χ')$ is ${\cal B}$-admissible. Let $t \in V(T'),$ such that $|χ'(t)| > k.$ Since $|S| \leq k,$ by definition, we have that $χ'(t) \cap V(C) \neq \emptyset,$ which also implies that $χ(t) \cap V(C) \neq \emptyset.$ Also, by \autoref{@dilaceraciones}, we have that $|χ(t)| \geq |χ'(t)| > k.$ Then since $(T, χ)$ is ${\cal B'}$-admissible and $χ(t) \cap V(C) \neq \emptyset,$ there must be some $B \in {\cal B}$ such that $χ(t) \cap Β = \emptyset.$ We show that $χ'(t) \cap B = \emptyset$ as well. Let $x \in χ'(t) \cap B.$ Let $t_{x} \in {\sf Trace}_{(T', χ')}(x)$ be closest to $s.$ Let $Z_{x} \subseteq T'$ be the unique $s\text{-}t_{x}$ shortest path in $T'.$ Suppose to the contrary that $χ'(t) \cap B \neq \emptyset.$ By definition of $χ'(t),$ it must be that $x \in S$ and $x \notin χ(t).$ Then this implies that $t \in V(Z_{x}).$ Since $|χ'(s)| \leq k,$ $t \neq s.$ Also, since $x \notin χ(t),$ $t \neq t_{x}.$ Then $t$ is an internal node of $T'.$ Moreover, since $χ(s)$ covers ${\cal B},$ we have that $χ(s) \cap B \neq \emptyset.$ Also by assumption, $χ(t_{x}) \cap B \neq \emptyset.$ Then, because of \autoref{@deliberately}, $χ(t) \cap B \neq \emptyset$ which contradicts the fact that $χ(t) \cap B = \emptyset.$
\end{proof}

\section{Edge-maximal graphs}
\label{@preoccupations}

In this section we prove a bound on the number of edges of a graph with strict bramble number at most $k.$ Recall that ${\cal G}_{k}=\{G\mid \sbn(G)\leq k\}.$ We achieve this by identifying the exact structure of all edge-maximal graphs of ${\cal G}_{k}.$ The striking difference with treewidth (where all $k$-trees on $n$ vertices have the same number of edges), is that, edge-maximal graphs of some specific size may have a varying number of edges. More formally, we do this by proving that the edge-maximal graphs of ${\cal G}_{k}$ are exactly the $k$-domino-trees.

\begin{theorem} 
\label{@sollicitations}
Let $G \in {\cal G}_{k}.$ $G$ is an edge-maximal graph if and only if, $G$ is a $k$-domino-tree.
\end{theorem}
\begin{proof}
Since $G \in {\cal G}_{k},$ by \autoref{@interrelated}, $G$ is a partial $k$-domino-tree. Assume that it is edge-maximal but not a $k$-domino-tree. This is an immediate contradiction, since we can still add edges to make it a $k$-domino-tree and remain in the class. For the converse assume that it is a $k$-domino-tree but not edge-maximal. Then, by \autoref{@interrelated}, $G$ is a proper spanning subgraph of a $k$-domino-tree, say $D.$ We prove that $G$ is isomorphic to $D,$ thus contradicting our assumption. If $V(G) \leq 2k,$ by definition and properties \ref{@presentiment}, \ref{@bezeichnenden}, $G$ and $D$ are isomorphic to $K_{|V(G)|}.$ Assume that $|V(G)| > 2k.$ Then $D$ has at least one minimal separator.

First observe that any minimal separator $S$ of $D$ is also a minimal separator of $G.$ Thus, since $G$ is a proper spanning subgraph of $D,$ for any connected component $C \in {\sf cc}(G - S),$ that contains some vertex of a connected component $C' \in {\sf cc}(D - S),$ it holds that $|V(C)| \leq |V(C')|.$ Let $K$ be a maximal clique in $D.$ If $G$ also contains $K$ as a maximal clique we are done. Assume otherwise.

First, assume that $K$ is an external maximal clique whose vertex set contains the external minimal separator $S \subseteq V(D)$ of $D.$ Let $C \in {\sf Acc}(G, S),$ such that $V(C)$ contains a vertex of $V(K) \setminus S.$ From the previous observation $|V(C)| \leq 2k.$ We distinguish two cases. Assume that $S$ is external in $G.$ If ${\sf cdeg}_{G}(S) = 2,$ then the union of the vertex sets of the two maximal cliques that contain $S,$ has size at most $2k,$ which contradicts property \ref{@presentiment}. Otherwise, in the context of $G,$ define ${\cal K}_{G}(S)$ as usual and assume that $|{\cal K}_{G}(S)| > 1.$ Then for any pair of different maximal cliques, say $K', K'' \in {\cal K}_{G}(S),$ since $|V(C)| \leq 2k,$ we have that ${\sf val}(K') + {\sf val}(K'') \leq k,$ which contradicts property \ref{@bezeichnenden}. Now, assume that $S$ is internal in $G.$ Then $C$ contains some other external minimal separator of $G,$ and we can reapply the two previous arguments.

Now, assume that $K$ is an internal maximal clique in $D$ and let $S, S' \subseteq V(D)$ be the two minimal separators covering $V(K).$ Since $K$ is not a maximal clique in $G,$ there exists a pair of vertices $x \in S, y \in S',$ that are not adjacent in $G.$ Then there exists a minimal $(x, y)$-separator $S'' \subseteq V(G)$ in $G$ contained in $V(K).$ Observe that, $S''$ cannot be an internal minimal separator with ${\sf cdeg}_{G}(S'') = 2,$ as that would contradict property \ref{@mandamientos}. Also, if ${\sf cdeg}_{G}(S'') > 2,$ then there exists a connected component of $G - S'',$ contained in $V(K),$ which implies that there exists at least one vertex of $V(K)$ that is not covered by $S \cup S'$ in $D.$ This contradicts property \ref{@precapitalist} in $D.$
\end{proof}

\paragraph{Bounds.} We continue by presenting a tight upper bound for the number of edges an edge-maximal graph in ${\cal G}_{k}$ can have. Consider $(T, χ)$ to be an extreme lenient tree decomposition of width $k$ for some graph $G.$ We proceed to count the maximum number of edges that $G$ can have. Recall, that in an extreme lenient tree decomposition all bags have size $k$ and no pair of bags is a subset of one another.

Root $T$ from some arbitrary internal node $ρ \in V(T).$ For every node $t \in V(T)$ different than $ρ,$ let $p \in V(T)$ be the parent of $t,$ i.e. the neighbor closest to $ρ.$ Then, define $k_{t} = |χ(t) \setminus χ(p)|,$ i.e. the number of new vertices of $G$ that this node introduces. Notice that, $k_{t} \in [k].$ Also for notational simplicity, let $X = V(T) \setminus \{ρ\}.$ In this way, it is clear that
\begin{align*}
&|V(G)| = k + \sum_{t \in X} k_{t},
\end{align*}
and for the number of edges,
\begin{align*}
|E(G)| &= \binom{k}{2} + \sum_{t \in X} \Bigg[ \binom{k_{t}}{2} + k \cdot k_{t} \Bigg]\\
&= \binom{k}{2} + \bigg(k - \frac{1}{2} \bigg)\bigg(|V(G)| - k\bigg) + \frac{1}{2} \sum_{t \in X} k_{t}^{2}.
\end{align*}

To maximize the above quantity, we have to maximize the sum of squares. We have that $\sum_{t \in X} k_{t} = |V(G)| - k$  and we want to partition this quantity into $|X|$ variables, $k_{t} \in [k],$ so as to maximize $\sum_{t \in X} k_{t}^{2}.$ The optimal solution is given by having as many variables equal to $k$ as possible. So, if $|V(G)| = c \cdot k + r,$ where $r < k,$ the optimal is given when  $k$ vertices 
are introduced for the root (this introduces $k\choose 2$ edges), $c-1$ nodes introduce $k$ new vertices each (this introduces $k^2+{k\choose 2}$ edges at a time) and we have a single node that introduces the remainder $r$ (this introduces $rk+ \binom{r}{2}$ edges). We have that:
\begin{eqnarray}
|E(G)| & \leq & \binom{k}{2} + \frac{c - 1}{2} \Big( 3k^{2} - k \Big) + rk + \binom{r}{2} \label{@distribuidos}
\end{eqnarray}
Replacing $c$ by $(|V(G)| - r)/k$ and $r$ by $|V(G)|~{\sf mod}~k,$ in \eqref{@distribuidos} we have that:
\begin{eqnarray}
|E(G)| & \leq & \frac{3k - 1}{2} |V(G)|- k^{2} - \frac{k}{2} \Big( |V(G)|~{\sf mod}~k \Big) + \frac{1}{2} \Big( |V(G)|~{\sf mod}~k \Big)^{2}  \label{@aubervilliers}
\end{eqnarray}

Let $G \in {\cal G}_{k}.$ Then there is a function $f : {\Bbb N} \times {\Bbb N} \to {\Bbb N}$ such that $|E(G)| \leq f(|V(G)|, k),$ where $f(|V(G)|, k),$ is defined as the right-hand upper bound  in \eqref{@aubervilliers}. This gives an upper bound to the number of edges of a graph $G$ where $\sbn(G)\leq k.$ 

Also this bound is tight. Consider a $k$-domino-tree $D,$ with $|V(D)| = |V(G)|,$ that is made up of $c$ linearly arranged maximal cliques, where one of the external maximal cliques has size $k + r,$ while all others have size $2k$ (see \autoref{@affectibility}). Then $|E(D)| = f(|V(D)|, k).$


\begin{figure}[t]
\begin{center}
\begin{tikzpicture}
\node (1) at (0,0) {};
\node (2) at (0,-1) {};
\node (3) at (1,0) {};
\node (4) at (1,-1) {};
\node (5) at (2,0) {};
\node (6) at (2,-1) {};
\node (7) at (3,0) {};
\node (8) at (3,-1) {};

\path (1) edge (2);
\path (1) edge (3);
\path (1) edge (4);
\path (2) edge (3);
\path (2) edge (4);
\path (3) edge (4);

\path (3) edge (4);
\path (3) edge (5);
\path (3) edge (6);
\path (4) edge (5);
\path (4) edge (6);
\path (5) edge (6);

\path (5) edge (6);
\path (5) edge (7);
\path (5) edge (8);
\path (6) edge (7);
\path (6) edge (8);
\path (7) edge (8);
\end{tikzpicture}
\end{center}
\caption{A $2$-domino-tree on $8$ vertices, and $16$ edges, which is the maximum possible number of edges a $2$-domino-tree on $8$ vertices may have.}
\label{@affectibility}
\end{figure}
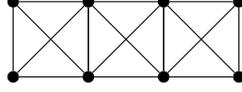

Now, assume that $|V(G)| \geq 3k$ and let $D'$ be a $k$-domino-tree, with $|V(D')| = |V(D)|,$ that is made up of $(c - 3)k + r + 2$ linearly arranged maximal cliques in a fan way, where all maximal cliques have $k-1$ vertices in common, both of the external maximal cliques have size $2k$ and all others have size $k + 1$ (see \autoref{@encareciendo}). Then:
\begin{eqnarray}
|E(D')| & = & \binom{k}{2} + ck^{2} + (r - 1)k \label{@antithetically}
\end{eqnarray}
Again, by replacing $c$ with $(|V(D')| - r)/k,$ in \eqref{@antithetically} we have that:
\begin{eqnarray}
|E(D')| & = & k|V(D')| + \frac{k^{2} - 3k}{2} \label{@universellement}
\end{eqnarray}

To conclude, this implies that the edge-maximal graphs  in \autoref{@sollicitations} do not necessarily have the same number of edges. Let $G \in {\cal G}_{k}$ be an edge-maximal graph with the maximum possible number of edges and let $G' \in {\cal G}_{k}$ be an edge-maximal graph with the minimum possible number of edges, such that $|V(G)| = |V(G')| = c \cdot k + r,$ where $c \geq 3.$ A lower bound on the size difference between the two edge sets is the following:
\begin{eqnarray}
|E(G)| - |E(G')| & \geq & \frac{c - 3}{2}  k(k-1) + \binom{r}{2}
\end{eqnarray}
or as a function of $n = |V(G)| = |V(G')|$:
\begin{eqnarray}
|E(G)| - |E(G')| & {\geq} & \frac{k - 1}{2} n - \frac{k}{2} \Big( n~{\sf mod}~k \Big) + \frac{1}{2} \Big( n~{\sf mod}~k \Big)^{2} - 3 \binom{k}{2} \label{@bereinstimmung}
\end{eqnarray}
Already for $k=2$ and $n = 8$ in (\ref{@bereinstimmung}), we have that $|E(G)| - |E(G')| \geq 1.$ \autoref{@affectibility} and \autoref{@encareciendo} depict this.

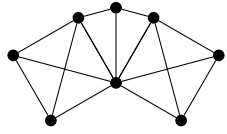
\begin{figure}[t]
\begin{center}
\begin{tikzpicture}

\node (1) at (-0.366,-0.634) {};
\node (2) at (0.134,-1.5) {};
\node (3) at (0.5,-0.134) {};

\node (6) at (1.5,-0.134) {};
\node (7) at (2.366,-0.634) {};
\node (8) at (1.866,-1.5) {};

\node (4) at (1,-1) {};
\node (5) at (1,0) {};

\path (1) edge (2);
\path (1) edge (3);
\path (1) edge (4);
\path (2) edge (3);
\path (2) edge (4);
\path (3) edge (4);

\path (3) edge (4);
\path (3) edge (5);
\path (4) edge (5);
\path (4) edge (6);
\path (5) edge (6);

\path (4) edge (6);
\path (4) edge (7);
\path (4) edge (8);
\path (6) edge (7);
\path (6) edge (8);
\path (7) edge (8);
\end{tikzpicture}
\end{center}
\caption{A $2$-domino-tree on $8$ vertices and $15$ edges.}
\label{@encareciendo}
\end{figure}

\section{Minor obstructions for strict bramble number at most two}

We use the term {\em graph collection} for finite sets of graphs, while for infinite sets of graphs we use the term {\em graph class}. Given a collection of graphs ${\cal H},$ we denote by $\exc({\cal H})$ the class of all graphs that do not contain any of the graphs in ${\cal H}$ as a minor.

A graph class ${\cal G}$ is {\em minor-closed} if every minor of a graph in ${\cal G}$ belongs to ${\cal G}.$ We define the set $\obs({\cal G})$ of all minor-minimal graphs not belonging to ${\cal G}.$ A direct consequence of the celebrated Robertson and Seymour Theorem \cite{RobertsonS04GMXX}, is that $\obs({\cal G})$ is a finite set. This implies that every minor-closed graph class ${\cal G}$ can be finitely characterised by this {\em obstruction set}, as a graph $H$ belongs to ${\cal G},$ if and only if, none of the (finitely many) graphs in $\obs({\cal G})$ is a minor of $H.$ Therefore ${\cal G}=\exc(\obs({\cal G})).$ Although the identification of the obstruction set of a minor-closed graph class can be a very difficult task (as the size of this set more often than not is enormous), there is an ever-growing list of characterizations (partial or complete)  of the obstruction sets for diverse minor-closed graph classes (see e.g., \cite{LeivaditisSSTTV20mino,Holst02onth,RobertsonST95sach,BodlaenderThil99,Thilikos00,ArnborgPC90forb,Archdeacon06akur,DinneenX02mino,KinnersleyL94,FioriniHJV17thee,DinneenV12obst,Wagner37uber}).

Recall that we have defined ${\cal G}_{k}$ to be the class of graphs with strict bramble number at most $k.$ The goal of this section is to give an alternative characterization of ${\cal G}_{2}$ in terms of {\em forbidden minors}. First, as we have already noticed, $\obs({\cal G}_{1})=\{K_3\},$ as ${\cal G}_{1}$ contains exactly the graphs in which each connected component is a tree. We start the study of $\obs({\cal G}_{2})$ with the following lemma.

\begin{lemma}\label{@redistribute} Every graph in $\obs({\cal G}_{2})$ is $2$-connected.
\end{lemma}
\begin{proof}
Let $G, G' \in {\cal G}_{2},$ such that they are disjoint. Let $x \in V(G),$ $y \in V(G')$ and let $G''$ be the graph, that is the result of the disjoint union of $G$ with $G',$ after identifying $x, y$ into a single vertex. Observe that $G'' \in {\cal G}_{2}.$ To see this, take any two disjoint lenient tree decompositions of $G$ and $G'$ and consider a new one:  the tree is created by 
adding an edge between the two nodes of the trees containing $x$ and $y$ and the set of  bags is the union of the set of bags of the two decompositions. This is a lenient tree decomposition of $G''$ of the same width. Now let $H \in \obs({\cal G}_{2})$ and assume it is not $2$-connected. Then there exists a cut-vertex that splits $H,$ into at least two proper minors of $H$ that belong to ${\cal G}_{2}.$ Then, by the previous argument, $H \in {\cal G}_{2},$ a contradiction.
\end{proof}

Let $G$ be a graph and let $(T, χ)$ be a tree decomposition of $G.$ The {\em adhesion set of an edge} $e = \{t, t'\} \in E(T)$ is the vertex set $χ(t) \cap χ(t')\subseteq V(G)$ and the {\em adhesion of an edge} $e \in E(T)$ is the size of the adhesion set of $e.$ The {\em adhesion} of the tree decomposition $(T, χ)$ is equal to the maximum adhesion of the edges of $T.$ The {\em adhesion variety} of a node $ t \in V(T),$ denoted by ${\sf adv}(t),$ is the number $|\{χ(t) \cap χ(t') \mid t' \in N_{T}(t)\}|,$ i.e. the number of different adhesion sets of edges incident to $t$ (see \autoref{fig_tree_dec}). For every $t \in V(T),$ we define the edge set $E(t) = \bigcup_{t' \in N_{T}(t)} {χ(t) \cap χ(t') \choose 2}$ and the {\em $t$-torso} of $(T, χ)$ as the graph $(χ(t), E(G[χ(t)]) \cup E(t)).$ We also call $(T, χ)$-{\em augmentation of} $G$ the graph obtained if we take the union of every $t$-torso of $G,$ $t \in V(T),$ and we denote it by ${\sf Aug}_{T,χ}(G),$ or simply ${\sf Aug}(G)$ when $({T,χ})$ is clear from the context. We call the edges of ${\sf Aug}_{T,χ}(G)$ that are not edges of $G,$ {\em $(T,χ)$-completion edges} (see \autoref{fig_tree_dec}). 

We also require the well-known notion of {\em triconnected decompositions}. The following result is a restatement, in our terminology, of the classic result of Tutte (see~\cite{Tutte66}). 

\begin{@menospreciando} \label{@menospreciando}
Every graph has a tree decomposition $(T, \chi)$ of adhesion at most two where every $t$-torso is either a 3-connected graph or a complete graph on at most three vertices. Moreover, for every $t\in V(T),$ and every completion edge $\{x,y\}$ of the $t$-torso, there is a $x\text{-}y$ path of length at least two in $G$ that does not contain edges of the $t$-torso and every two such paths, corresponding to the same torso, are internally disjoint.
\end{@menospreciando}

We call a decomposition $(T, χ)$ of a graph $G,$ as in \autoref{@menospreciando}, a {\em triconnected decomposition} of $G$ and its $t$-torsos {\em triconnected components} of $G.$ Observe that every triconnected component of $G$ is a minor of $G.$

We also require the following easy proposition. It follows easily by the classic result of Tutte in~\cite{Tutte61athe}, asserting that for every 3-connected graph $G$ there is a sequence $G_{1},\ldots,G_{r}$
of 3-connected graphs such that $G_{i}$ is a minor of $G_{i+1},$ for every $i\in[r-1]$ and where $G_{r}=G$ and $G_{1}$ is a wheel graph.

\begin{figure}[t]
\begin{center}
\graphicspath{{./Figures/}}
\scalebox{1}{\includegraphics{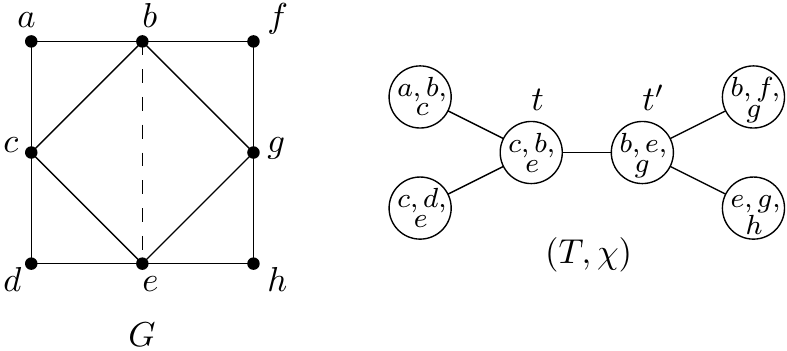}}
\end{center}
\caption{On the right we can see a tree-decomposition $(T, χ)$ of $G$ (we have marked the vertices of each bag $χ(t)$). Notice that the adhesion set of $\{t,t'\} \in E(T)$ is the set $\{b,e\}$ and that the adhesion of $(T, χ)$ is $2.$ Also notice that the $t$-torso is the graph $G[\{c,b,e\}]$ with the addition of the $(T,χ)$-completion edge $\{b,e\}.$ Additionally, ${\sf adv}(t) = 3$ and ${\sf Aug}_{T, χ}(G)$ is $G$ with the addition of the edge $\{b,e\}.$}
\label{fig_tree_dec}
\end{figure}

\begin{@menospreciando}\label{@testiguienne}
Every 3-connected graph not containing $W_{4}$ as a minor is isomorphic to $K_{4}.$
\end{@menospreciando}

Let ${\cal Z}$ be the collection of graphs $W_{4},$ $H_{1},$ $H_{2},$ as seen in \autoref{fig_obs}. We prove the following Lemmata.

\begin{lemma}\label{@menschengeist} Let $G$ be a $2$-connected graph that excludes as a minor the graphs in ${\cal Z}.$ Then there is a triconnected decomposition $(T, χ),$ such that:
\begin{itemize}
\item every edge of $T$ has adhesion two;
\item every node of $T$ has adhesion variety at most two;
\item for every  $t \in V(T)$ with adhesion variety two, the set  $χ(t)$ is the union of the adhesion sets of the edges incident to $t$ in $T.$
\end{itemize}
\end{lemma}
\begin{proof}
Let $(T, χ)$ be a triconnected decomposition of $G,$ which subject to \autoref{@menospreciando} minimizes,
\begin{align*}
&μ(Τ,χ) := \sum_{t \in V(T)} ({\sf adv}(t) - 2).
\end{align*}
Since $G$ is $2$-connected we can assume that every edge of $T$ has adhesion exactly two. Observe that, since $W_4$ is not a minor of $G,$ by \autoref{@testiguienne}, every $3$-connected $t$-torso of $(T, χ)$ is isomorphic to $K_{4}.$ Also, since $(T, χ)$ has adhesion two, every other triconnected component of $G$ must be isomorphic to $K_3.$ Therefore, for every $t\in V(T),$ $|χ(t)|\in\{3,4\}.$ For every node $t \in V(T),$ let ${\cal A}_{t}$ be the family of different adhesion sets of the edges incident to $t$ in $T.$ Notice that, $|{\cal A}_{t}| = {\sf adv}(t).$ For each $A_{i} \in {\cal A}_{t},$ $i \in [{\sf adv}(t)],$ we consider some $C_{i} \in {\sf Acc}(G,A_{i}),$ where $χ(t) \not\subseteq V(C_{i}).$ We distinguish two cases.\medskip

\noindent{\em Case 1:} $G[χ(t)]$ is isomorphic to $K_{3}.$ Assume that $|{\cal A}_{t}| \geq 3.$ Then we distinguish two subcases.\smallskip

\noindent{\em Subcase 1a:} For some $i \in [3],$ $C_{i}$ is a path between the two vertices $x, y \in A_{i}.$ This implies that $A_{i} \notin E(G)$ and that $C_{i}$ has length at least $2.$ Let $z \neq x, y$ be the third vertex of $χ(t).$ Then we can add a completion edge between $z$ and every vertex of $C_{i}$ and then replace node $t$ with a path in $T,$ for each new triangle created. Clearly, this new triconnected decomposition decreases $μ(T,χ),$ which is a contradiction.\smallskip

\noindent{\em Subcase 1b:}  For every $i \in [3],$ $C_{i}$ is not a path between the two vertices $x, y \in A_{i}.$ Then, since $G$ is $2$-connected, $C_{i}$ contains a cycle. Then we can contract $C_{i}$ into a triangle with $x, y$ as one side, and we obtain $H_{1}$ as a minor, which contradicts our hypothesis.\medskip

\noindent{\em Case 2:} $G[χ(t)]$ is isomorphic to $K_{4}.$ We prove that for every pair of different $A_{i}, A_{j} \in {\cal A}_{t},$ $A_{i} \cap A_{j} = \emptyset.$ Notice that this implies that $|{\cal A}_{t}| \leq 2$ and that if $|{\cal A}_{t}| = 2,$ $χ(t) = \bigcup {\cal A}_{t}.$ Assume to the contrary, that $A_{i} \cap A_{j} \neq \emptyset,$ for some pair of different $A_{i}, A_{j} \in {\cal A}_{t}.$ Then contract $C_{i}, C_{j}$ into single vertices and remove $A_{i}, A_{j}$ from $E(G),$ if the edges exist. Then we get $H_{2}$ as a minor of $G,$ which again contradicts our hypothesis.
\end{proof}

\begin{lemma}\label{@uncomprehended} Let $G$ be a $2$-connected graph that excludes as a minor the graphs in ${\cal Z}.$ Then $G \in {\cal G}_{2}.$
\end{lemma}
\begin{proof}
We prove that $G$ admits a lenient tree decomposition $(T, χ)$ of width $2.$ Then, by \autoref{@interrelated}, $G \in {\cal G}_{2}.$ Let $(T', χ')$ be a triconnected decomposition of $G,$ as in \autoref{@menschengeist}. We obtain $(T, χ)$ from $(T', χ')$ as follows. We define the set, ${\cal A} = \{ χ'(t) \cap χ'(t') \mid \{t, t' \} \in E(T')\},$ i.e. the set of different adhesion sets of $(T', χ').$ We define a node $t \in V(T)$ for each $A \in {\cal A}$ and we set $χ(t) = A.$ Two nodes $t, t' \in V(T)$ are adjacent if there exists a node $t'' \in V(T')$ such that, $χ(t)$ and  $χ(t'),$ are the adhesion sets of edges incident to $t''.$ Also for every leaf node $t \in V(T')$ of $T',$ we put all vertices of $χ'(t),$ which are not contained in some adhesion set, in a new node in $T$ which is adjacent to the unique node of $T,$ corresponding to the adhesion set of the edge incident to $t$ in $T'.$ Since every node of $(T', χ')$ has adhesion variety at most two, clearly $T$ is a tree. Since the adhesion sets of edges incident to a node $t$ in $T'$ with adhesion variety two, cover every vertex of $χ'(t),$ a vertex is either in a leaf or in an adhesion set in $T',$ thus $(T, χ)$ satisfies conditions \ref{@predominantly} and \ref{@calamitously}. Now, consider a vertex $x \in V(G).$ Since in $T',$ ${\sf Trace}_{(T', χ')}(x)$ is connected, the nodes corresponding to the different adhesion sets that contain $x$ will also be connected in $(T, χ),$ thus condition \ref{@unpretentious} is also satisfied. Since $(T', χ')$ has adhesion two, the width of $(T, χ)$ is two.
\end{proof}

\begin{figure}[t]
\begin{center}
\graphicspath{{./Figures/}}
\scalebox{1}{\includegraphics{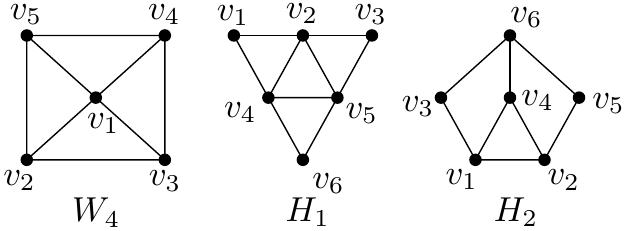}}
\end{center}
\caption{The graphs in ${\cal Z}.$}
\label{fig_obs}
\end{figure}

\begin{lemma}\label{@disjointedly} Let $G$ be a graph that has no strict bramble of order greater than $2.$ Then $G$ excludes as a minor the graphs in ${\cal Z}.$
\end{lemma}
\begin{proof}
We provide for each of the graphs in ${\cal Z},$ a strict bramble of order three. For $W_{4}$ we have the following strict bramble,
\begin{align*}
&\{\{v_1, v_2, v_3 \}, \{v_1, v_3, v_4 \}, \{v_1, v_4, v_5 \}, \{v_1, v_2, v_5 \}, \{v_2, v_3, v_5 \}, \{v_2, v_3, v_4 \}, \{v_3, v_4, v_5 \},\\
&\{v_2, v_4, v_5 \}, \{v_1, v_2, v_4 \}, \{v_1, v_3, v_5 \}\}
\end{align*}
for $H_{1}$ we have the following strict bramble,
\begin{align*}
&\{\{v_1, v_2, v_3 \}, \{v_3, v_5, v_6 \}, \{v_1, v_4, v_6 \}, \{v_2, v_4, v_5 \}, \{v_2, v_3, v_4 \}, \{v_1, v_2, v_5 \}, \{v_1, v_4, v_5 \}\}
\end{align*}
and for $H_{2}$ we have the following strict bramble.
\begin{align*}
&\{\{v_1, v_2, v_3 \}, \{v_1, v_2, v_5 \}, \{v_1, v_3, v_4 \}, \{v_2, v_4, v_5 \}, \{v_1, v_4, v_6 \}, \{v_2, v_4, v_6 \}, \{v_3, v_5, v_6 \}\}
\end{align*}

By the equivalency of \autoref{@interrelated}, it holds that $\sbn$ is minor-closed and thus $G$ excludes as a minor the graphs in ${\cal Z}.$
\end{proof}

We now present the main theorem of this section.

\begin{theorem} The obstruction set of the class ${\cal G}_{2},$ consists of the graphs in ${\cal Z}.$
\end{theorem}
\begin{proof} Observe that any proper minor of the graphs in ${\cal Z}$ has strict bramble number at most two. Combined with \autoref{@disjointedly}, this implies that ${\cal Z} \subseteq {\sf Obs}({\cal G}_{2}).$ Now assume that there exists an obstruction $Z \in {\sf Obs}({\cal G}_{2}) \setminus {\cal Z}.$ As $Z\in  {\sf Obs}({\cal G}_{2}),$  $Z$ excludes all the graphs in ${\cal Z}$ as a minor. Additionally, by \autoref{@redistribute}, $Z$ is $2$-connected. Then, by \autoref{@uncomprehended}, $Z \in {\cal G}_{2},$ which contradicts the choice of $Z.$ Thus ${\sf Obs}({\cal G}_{2}) = {\cal Z}.$
\end{proof}

Notice that ${\sf Obs}({\cal G}_{2}) = {\cal Z},$ indicates that the class ${\cal G}_{2}$ is ``orthogonal'' to the class, say ${\cal T}_2,$  of the graphs of treewidth at most two, where $\obs({\cal T}_{2})=\{K_{4}\}.$ Indeed $\sbn(K_{4})=2$ and  $\tw(K_{4})=3,$ while $\sbn(H_{1})=3$ and  $\tw(H_{1})=2.$

\section{\NP-completeness}
\label{@unflattering}

In this section, we prove that given a graph $G$ and an integer $k \in {\Bbb N},$ deciding whether $\sbn(G) \leq k$ or, equivalently, $\ltp(G) \leq k$ is \NP-complete. Membership in {\sf NP} follows from the definition of $\sbn$
via lenient tree decompositions, due to the min-max equivalence of~\autoref{@interrelated}. For the {\sf NP}-hardness, we reduce the problem of deciding whether $\tw(G)\leq k$ (that is {\sf NP}-complete~\cite{ArnborgCP87comp}) to our problem. We first prove the following Lemma.

\begin{lemma}\label{@comparability} Let $H$ be a graph obtained from a graph $G$ after replacing every edge of $G$ with $2k-1$ paths of length two. If $H$ admits a lenient tree decomposition $(T, χ)$ of width at most $k$ then, for every pair of vertices $x, y \in V(G)$ that are adjacent in $G,$ there is a bag of $(T, χ)$ that contains them both.
\end{lemma}
\begin{proof}
Let $x, y \in V(G)$ be an adjacent pair of vertices in $G.$ We can easily observe that, by the definition of $H,$ there are $2k - 1$ disjoint $x\text{-}y$ paths in $H.$ Then, by Menger's Theorem, any $(x,y)$-separator in $H$ has size at least $2k - 1.$ Let $t_{x}, t_{y} \in V(T)$ be the closest nodes in $(T, χ)$ whose bags contain $x, y$ respectively. We want to show that the distance between them is zero. Assume that it is at least two. Then there exists a node of $T,$ in between $t_{x}$ and $t_{y}$ whose bag, by \autoref{@metaphysicians}, is a $(x,y)$-separator of $H$ of size at most $k$ which is a contradiction. Now assume that the distance is one. Then since $|χ(t_{x}) \cup χ(t_{y})| \leq 2k,$ there is a neighbour of $x, y$ in $H$ which is not contained in either $χ(t_{x})$ or $χ(t_{y}),$ which violates condition \ref{@calamitously}.
\end{proof}

We continue with the proof of the reduction.

\begin{theorem} There exists a polynomially computable function that, given a graph $G$ and an integer $k \in {\Bbb N},$ outputs a graph $H$ such that, $\tw(G) \leq k - 1$ if and only if, $\ltp(H) \leq k.$
\end{theorem}
\begin{proof} We obtain $H$ by replacing every edge of $G$ with $2k - 1$ paths of length two. Observe that the forward direction easily holds. Indeed, having parallel edges and subdividing them does not increase the tree width of a graph. Then the claim follows since $\tw(G) \leq \ltp(G).$ For the backwards direction, by \autoref{@interrelated}, it is equivalent to argue in terms of lenient tree decompositions. Let $(T, χ)$ be a lenient tree decomposition of $H$ of width at most $k-1.$ Then, by \autoref{@comparability}, for every pair of vertices $x, y \in V(H)$ adjacent in $G,$ there is a bag of $(T, χ)$ that contains them both. Now simply observe that $(T, χ)$ is a tree decomposition of $G,$ of width at most $k-1.$ Thus $\tw(G) \leq k-1.$
\end{proof}

\section{Open problems}
\label{@mechanisches}

In this paper we initiated the study of the strict bramble number. As computing this parameter is an {\sf NP}-complete problem, it is meaningful to design a parameterized algorithm, i.e. one that can answer whether $\sbn(G) \leq k$ in $f(k) \cdot |V(G)|^{O(1)}$ steps, for some computable function $f.$ Such an algorithm actually exists because this question is equivalent to checking whether none of the graphs in $\obs(G_{k})$ is contained in $G$ as a minor, $\obs(G_{k})$ is finite for every $k$ because of the Robertson and Seymour theorem \cite{RobertsonS04GMXX}, and minor checking for graphs of bounded treewidth can be done in linear (on the size of the host graph) time. However, this argument is not constructive as we do not know $\obs(G_{k}).$ Therefore it is an open problem to actually design such an algorithm. A possible step for this would to identify $\obs(G_{k})$ for higher values of $k.$ However, this seems to be a hard problem, even for $k=3.$ 
Combining computer search with some graph-theoretic observations, we were able to identify\footnote{See 
\href{https://www.cs.upc.edu/~sedthilk/twointer/Obstruction_checker.py}{https://www.cs.upc.edu/\~{}sedthilk/twointer/Obstruction\_checker.py}
for the verification code.} at least 194 members of $\obs({\cal G}_{3})$: %
\href{https://www.cs.upc.edu/~sedthilk/twointer/7%20vertices.zip}{\greenfbox{5 with 7 vertices}}, 
\href{https://www.cs.upc.edu/~sedthilk/twointer/8%20vertices.zip}{\greenfbox{19 with 8 vertices}}, 
\href{https://www.cs.upc.edu/~sedthilk/twointer/9%20vertices.zip}{\greenfbox{75 with 9 vertices}}, 
\href{https://www.cs.upc.edu/~sedthilk/twointer/10%20vertices.zip}{\greenfbox{86 with 10 vertices}}, and 
\href{https://www.cs.upc.edu/~sedthilk/twointer/11%20vertices.zip}{\greenfbox{at least 9 with at least $11$ vertices}}.
%


\begin{thebibliography}{10}

\bibitem{AidunDMYY19treew}
Ivan Aidun, Frances Dean, Ralph Morrison, Teresa Yu, and Julie Yuan.
\newblock Treewidth and gonality of glued grid graphs, 2019.
\newblock \href {http://arxiv.org/abs/1808.09475} {\path{arXiv:1808.09475}}.

\bibitem{Archdeacon06akur}
Dan Archdeacon.
\newblock {A Kuratowski theorem for the projective plane}.
\newblock {\em Journal of Graph Theory}, 5:243--246, 10 2006.
\newblock \href {https://doi.org/10.1002/jgt.3190050305}
  {\path{doi:10.1002/jgt.3190050305}}.

\bibitem{ArnborgCP87comp}
Stefan Arnborg, Derek~G. Corneil, and Andrzej Proskurowski.
\newblock Complexity of finding embeddings in a $k$-tree.
\newblock {\em SIAM Journal on Algebraic Discrete Methods}, 8(2):277–284,
  1987.
\newblock \href {https://doi.org/10.1137/0608024} {\path{doi:10.1137/0608024}}.

\bibitem{ArnborgPC90forb}
Stefan Arnborg, Andrzej Proskurowski, and Derek~G. Corneil.
\newblock Forbidden minors characterization of partial 3-trees.
\newblock {\em Discrete Mathematics}, 80(1):1--19, 1990.
\newblock \href {https://doi.org/10.1016/0012-365X(90)90292-P}
  {\path{doi:10.1016/0012-365X(90)90292-P}}.

\bibitem{BellenbaumD02twosh}
Patrick Bellenbaum and Reinhard Diestel.
\newblock Two short proofs concerning tree-decompositions.
\newblock {\em Comb. Probab. Comput.}, 11(6):541--547, 2002.
\newblock \href {https://doi.org/10.1017/S0963548302005369}
  {\path{doi:10.1017/S0963548302005369}}.

\bibitem{BerteleB72nonser}
Umberto Bertelé and Francesco Brioschi, editors.
\newblock {\em Nonserial Dynamic Programming}, volume~91 of {\em Mathematics in
  Science and Engineering}.
\newblock Elsevier, 1972.
\newblock URL:
  \url{https://www.sciencedirect.com/science/article/pii/S007653920860140X},
  \href {https://doi.org/https://doi.org/10.1016/S0076-5392(08)60140-X}
  {\path{doi:https://doi.org/10.1016/S0076-5392(08)60140-X}}.

\bibitem{Bodlaender98apart}
Hans~L. Bodlaender.
\newblock A partial {$k$}-arboretum of graphs with bounded treewidth.
\newblock {\em Theoret. Comput. Sci.}, 209(1-2):1--45, 1998.

\bibitem{BodlaenderThil99}
Hans~L. Bodlaender and Dimitrios~M. Thilikos.
\newblock Graphs with branchwidth at most three.
\newblock {\em Journal of Algorithms}, 32(2):167--194, 1999.
\newblock \href {https://doi.org/10.1006/jagm.1999.1011}
  {\path{doi:10.1006/jagm.1999.1011}}.

\bibitem{DinneenV12obst}
Michael~J. Dinneen and Ralph Versteegen.
\newblock Obstructions for the graphs of vertex cover seven.
\newblock Technical Report CDMTCS-430, University of Auckland, 2012.
\newblock URL: \url{http://hdl.handle.net/2292/22193}.

\bibitem{DinneenX02mino}
Michael~J. Dinneen and Liu Xiong.
\newblock Minor-order obstructions for the graphs of vertex cover 6.
\newblock {\em Journal of Graph Theory}, 41(3):163--178, 2002.
\newblock \href {https://doi.org/10.1002/jgt.10059}
  {\path{doi:10.1002/jgt.10059}}.

\bibitem{FioriniHJV17thee}
Samuel Fiorini, Tony Huynh, Gwena{\"{e}}l Joret, and Antonios Varvitsiotis.
\newblock The excluded minors for isometric realizability in the plane.
\newblock {\em {SIAM} Journal on Discrete Mathematics}, 31(1):438--453, 2017.
\newblock \href {https://doi.org/10.1137/16M1064775}
  {\path{doi:10.1137/16M1064775}}.

\bibitem{Halin76sfun}
Rudolf Halin.
\newblock S-functions for graphs.
\newblock {\em Journal of Geometry}, 8(1):171--186, 1976.
\newblock \href {https://doi.org/10.1007/BF01917434}
  {\path{doi:10.1007/BF01917434}}.

\bibitem{HarveyW17param}
Daniel~J. Harvey and David~R. Wood.
\newblock Parameters tied to treewidth.
\newblock {\em J. Graph Theory}, 84(4):364--385, 2017.
\newblock \href {https://doi.org/10.1002/jgt.22030}
  {\path{doi:10.1002/jgt.22030}}.

\bibitem{KinnersleyL94}
Nancy~G. Kinnersley and Michael~A. Langston.
\newblock Obstruction set isolation for the gate matrix layout problem.
\newblock {\em Discrete Applied Mathematics}, 54(2):169--213, 1994.
\newblock \href {https://doi.org/10.1016/0166-218X(94)90021-3}
  {\path{doi:10.1016/0166-218X(94)90021-3}}.

\bibitem{KozawaOY14lower}
Kyohei Kozawa, Yota Otachi, and Koichi Yamazaki.
\newblock Lower bounds for treewidth of product graphs.
\newblock {\em Discret. Appl. Math.}, 162:251--258, 2014.
\newblock \href {https://doi.org/10.1016/j.dam.2013.08.005}
  {\path{doi:10.1016/j.dam.2013.08.005}}.

\bibitem{LeivaditisSSTTV20mino}
Alexandros Leivaditis, Alexandros Singh, Giannos Stamoulis, Dimitrios~M.
  Thilikos, Konstantinos Tsatsanis, and Vasiliki Velona.
\newblock Minor-obstructions for apex sub-unicyclic graphs.
\newblock {\em Discrete Applied Mathematics}, 284:538--555, 2020.
\newblock \href {https://doi.org/10.1016/j.dam.2020.04.019}
  {\path{doi:10.1016/j.dam.2020.04.019}}.

\bibitem{Reed97anewc}
Bruce~A Reed.
\newblock {\em Tree Width and Tangles: A New Connectivity Measure and Some
  Applications}, page 87–162.
\newblock London Mathematical Society Lecture Note Series. Cambridge University
  Press, 1997.
\newblock \href {https://doi.org/10.1017/CBO9780511662119.006}
  {\path{doi:10.1017/CBO9780511662119.006}}.

\bibitem{RobertsonS84GMIII}
N.~Robertson and Paul~D. Seymour.
\newblock {Graph Minors. III. Planar tree-width}.
\newblock {\em Journal of Combinatorial Theory, Series B}, 36(1):49--64, 1984.

\bibitem{RobertsonS04GMXX}
Neil Robertson and Paul~D. Seymour.
\newblock {Graph Minors. XX. Wagner's conjecture}.
\newblock {\em Journal of Combinatorial Theory, Series B}, 92(2):325--357,
  2004.

\bibitem{RobertsonST95sach}
Neil Robertson, Paul~D. Seymour, and Robin Thomas.
\newblock Sachs' linkless embedding conjecture.
\newblock {\em Journal of Combinatorial Theory, Series B}, 64(2):185--227,
  1995.
\newblock \href {https://doi.org/10.1006/jctb.1995.1032}
  {\path{doi:10.1006/jctb.1995.1032}}.

\bibitem{SeymourT93graph}
Paul~D. Seymour and Robin Thomas.
\newblock Graph searching and a min-max theorem for tree-width.
\newblock {\em J. Comb. Theory Ser. B}, 58(1):22--33, 1993.

\bibitem{Thilikos00}
Dimitrios~M. Thilikos.
\newblock Algorithms and obstructions for linear-width and related search
  parameters.
\newblock {\em Discrete Applied Mathematics}, 105(1):239--271, 2000.
\newblock \href {https://doi.org/10.1016/S0166-218X(00)00175-X}
  {\path{doi:10.1016/S0166-218X(00)00175-X}}.

\bibitem{Tutte61athe}
William~T. Tutte.
\newblock A theory of {$3$}-connected graphs.
\newblock {\em Nederl. Akad. Wetensch. Proc. Ser. Indag. Math.}, 23:441--455,
  1961.

\bibitem{Tutte66}
William~T. Tutte.
\newblock {\em Connectivity in graphs}.
\newblock University of Toronto Press, 1966.

\bibitem{Holst02onth}
Hein van~der Holst.
\newblock On the ``largeur d'arborescence''.
\newblock {\em Journal of Graph Theory}, 41(1):24--52, 2002.
\newblock \href {https://doi.org/10.1002/jgt.10046}
  {\path{doi:10.1002/jgt.10046}}.

\bibitem{Wagner37uber}
Klaus Wagner.
\newblock {\"U}ber eine eigenschaft der ebenen komplexe.
\newblock {\em Mathematische Annalen}, 114:570--590, 1937.
\newblock \href {https://doi.org/10.1007/BF01594196}
  {\path{doi:10.1007/BF01594196}}.

\end{thebibliography}

\end{document}